\documentclass[final,onethmnum]{siamltex}

\usepackage{amssymb}
\usepackage{amsmath}
\usepackage{graphicx,hyperref,subfigure}
\usepackage{nicefrac}

\newtheorem{thm}{Theorem}
\newtheorem{lem}{Lemma}
\newtheorem{rem}{Remark}
\newtheorem{cor}{Corollary}
\newtheorem{mdef}{Definition}
\newtheorem{prop}{Proposition}

\def\eps{\varepsilon}
\def\Id{{\rm Id}}
\def\op{{\rm op}}
\def\opaw{{\rm op^{AW}}}
\def\opwe{{\rm op^{We}}}
\def\tr{{\rm tr}}
\def\T{{\rm T}}
\def\C{\mathbb C}
\def\N{\mathbb N}
\def\P{\mathbb P}
\def\R{\mathbb R}
\def\H{{\mathcal H}}
\def\W{{\mathcal W}}

\title{Propagation of Quantum Expectations\\ with Husimi Functions}
\author{Johannes Keller
\thanks{Zentrum Mathematik, Technische Universit\"at M\"unchen, Boltzmannstra\ss e 3, 
85748 Garching bei M\"unchen, Germany ({\tt keller@ma.tum.de}) }
\and Caroline Lasser \thanks{Zentrum Mathematik, Technische Universit\"at M\"unchen, Boltzmannstra\ss e 3, 
85748 Garching bei M\"unchen, Germany ({\tt classer@ma.tum.de}) }.}
\date{\today}

\begin{document}
\maketitle

\begin{abstract}
We analyse the dynamics of expectation values of quantum 
observables for the time-dependent semiclassical Schr\"odinger equation.
To benefit from the positivity of Husimi functions, we switch
 between observables obtained from Weyl  and Anti-Wick quantization. 
We develop and prove a second order Egorov type propagation 
theorem with Husimi functions by establishing transition and commutator rules for  
Weyl and Anti-Wick operators. We provide a discretized version of our
 theorem and present numerical experiments for Schr\"odinger equations 
in dimensions two and six that validate our results.
\end{abstract}



\section{Introduction}
Our investigation is devoted to the time-dependent Schr\"odin\-ger equation
\begin{equation}
\label{eq:schro}
i\eps\partial_t\psi =-\tfrac{\eps^2}{2}\Delta\psi + V\psi,\qquad \psi(0)=\psi_0
\end{equation}
for general square integrable initial data $\psi_0:\R^d\to\C$ with $\|\psi_0\|_{L^2}=1$. We assume that the potential $V:\R^d\to\R$ is 
smooth and satisfies suitable growth conditions at infinity, such that the semi-classical Schr\"odinger operator 
$$
H=-\tfrac{\eps^2}{2}\Delta+V
$$ 
is essentially selfadjoint. Such potentials are met in molecular quantum dynamics for the effective description of nuclear motion after performing 
the time-dependent Born-Oppenheimer approximation \cite[Theorem~1]{ST01}. In this application, the parameter $\eps$ is the square root of the ratio of electronic 
versus average nuclear mass. Hence, $\eps>0$ is small and typically ranges between $\nicefrac{1}{1000}$ and $\nicefrac{1}{10}$. 
This implies that the solution of~(\ref{eq:schro}) is highly oscillatory in space and time with frequencies of the order~$\nicefrac1\eps$. However, 
the time evolution of expectation values
$$
t\mapsto\langle\psi_t,A\psi_t\rangle_{L^2} = \int_{\R^d} \overline{\psi_t(x)} A\psi_t(x) dx
$$
is less oscillatory, and the Egorov theorem \cite[Theorem~1.2]{BR02} suggests a semiclassical approximation using the Hamiltonian flow 
$\Phi^t:\R^{2d}\to\R^{2d}$ of the classical equations of motion $\dot q=p$, $\dot p=-\nabla V(q)$.

For this approximation we adopt the phase space point of view and consider operators $A=\opwe(a)$, which are obtained from functions $a:\R^{2d}\to\R$ by the Weyl quantization. Consequently, the Schr\"odinger operator $H=\op_\eps(h)$ is seen as the Weyl quantization of the Hamilton function 
$$
h:\R^{2d}\to\R,\qquad h(q,p)=\tfrac12|p|^2+V(q). 
$$
Expectation values of Weyl quantized operators can be written as phase space integrals
$$
\langle\psi_t,\opwe(a)\psi_t\rangle_{L^2} = \int_{\R^{2d}} a(z) \W^\eps(\psi_t)(z) dz,
$$
where $\W^\eps(\psi_t):\R^{2d}\to\R$ is the Wigner function of the square integrable function~$\psi_t:\R^d\to\C$. The Egorov theorem can be expressed in terms of Wigner functions as 
\begin{equation}
\label{eq:eg}
\langle\psi_t,\opwe(a)\psi_t\rangle_{L^2} = \int_{\R^{2d}} (a\circ\Phi^t)(z)\, \W^\eps(\psi_0)(z)\, dz + O(\eps^2). 
\end{equation}

In general, Wigner functions attain negative values. However, a proper smoothing by a Gaussian function $G_{\eps/2}$ with mean zero and covariance $\frac\eps2\,{\rm Id}$ results in a nonnegative function
$$
\H^\eps(\psi_t) = \W^\eps(\psi_t)* G_{\eps/2},
$$
the so-called Husimi function \cite[Remark~1.4]{GMMP97}. Hence, contrary to the Wigner function, the Husimi function can be viewed as a probability distribution on phase space, which entails the straightforward application of Monte Carlo type methods for sampling, especially in high dimensions. However, the gain in positivity is accompanied by a loss in accuracy with respect to time propagation, since replacing the Wigner function~$\W^\eps(\psi_0)$ by the Husimi function~$\H^\eps(\psi_0)$ in the Egorov theorem (\ref{eq:eg}) deteriorates the approximation error from second to first order in $\eps$. This motivates the question, whether an $\eps$-corrected flow can compensate this accuracy loss when propagating Husimi functions. 

Our main result answers the question positively. Its basic constituents are the $\eps$-dependent modification 
$$
a_\eps=a-\tfrac\eps4\Delta a
$$ 
of a phase space function $a:\R^{2d}\to\R$ by subtracting its phase space Laplacian, the flow $\Phi^t_\eps:\R^{2d}\to\R^{2d}$ associated with the Hamilton function $h_\eps=h-\tfrac\eps4\Delta:\R^{2d}\to\R$,
$$
h_\eps(q,p) = \tfrac{1}{2}|p|^2 + V(q) -\tfrac{\eps}{4}\left(d+\Delta V(q)\right),
$$
together with the solutions 
$$
\Lambda^t_\eps\in\R^{2d\times2d},\qquad \Gamma^t_\eps\in\R^{2d}
$$ 
of two first order ordinary differential equations, whose right hand side is  
built from evaluations of $D^2h$ and $D^3 h$ along $\Phi^t_\eps$, respectively. We obtain:

\begin{thm} There exists a constant $C=C(a,V,t)>0$ such that for all square integrable initial data $\psi_0:\R^d\to\C$ with $\|\psi_0\|_{L^2}=1$ and all $\eps>0$ the solution of the Schr\"odinger equation (\ref{eq:schro}) satisfies 
$$
\left|\langle\psi_t,\opwe(a)\psi_t\rangle_{L^2} - \int_{\R^{2d}} F^t_\eps(a_\eps)(z)\, \H^\eps(\psi_0)(z)\, dz \right| \le C \eps^2 
$$
with 
$$
F^t_\eps(a_\eps) = a_\eps\circ\Phi^t_\eps - \tfrac\eps2 \left(\Gamma^t_\eps\cdot\left(\nabla a\circ\Phi^t_\eps\right) + \tr\left(\Lambda^t_\eps \left(D^2 a\circ\Phi^t_\eps\right)\right)\right).
$$
\end{thm}

The method of proof relies on a systematic transition between Weyl and Anti-Wick quantization combined with an $\eps$-expansion of commutators in Weyl quantization. Therefore, the generalization of the approximation to errors of the order $\eps^k$, $k\ge2$, is possible. 

The theorem allows an algorithmic realization in the spirit of the particle method developed in \cite{LR10}, which is suitable for the fast computation of expectation values in high dimensions: sample the initial Husimi function $\H^\eps(\psi_0)$, discretize the ordinary differential equations for $\Phi^t_\eps$, $\Gamma_t^\eps$ and $\Lambda_t^\eps$, propagate the sample points along the discretized flow, and compute the expectation value at time $t$ by averaging over $F^t_\eps(a)$-evaluations for the sample points. Our numerical experiments for Schr\"odinger equations in dimensions two and six, one with a torsional potential the other with a Henon-Heiles potential, indicate that both the initial sampling and the flow discretizations can acchieve sufficient accuracy to confirm the approximation's asymptotic convergence rate.  
     
\subsection{Related research}

\cite{LR10} has developed a discrete version of the Egorov theorem by propagating samples drawn from the initial Wigner function along the classical flow $\Phi^t$. The negativity of the Wigner function has been accounted for by stratified sampling as well as sampling from the uniform distribution on the effective support of the Wigner function. The observation, that the approximation order deteriorates from second to first order in $\eps$ when merely replacing the Wigner by the Husimi function, has been formulated in \cite[Proposition~1]{KLW09} in the context of surface hopping algorithms. 

If the aim is not just the approximation of expectation values but of the wave function itself, then semiclassics have been successfully applied as well. The time-splitting algorithm of \cite{FGL09} uses a Galerkin approximation with Hagedorn wave packets, which are products of polynomials with a Gaussian function of time-varying first and second moments. Related approximations are the Gaussian beam methods, whose convergence for the Schr\"odinger equation, the acoustic wave equation as well as  strictly hyperbolic systems has been proven in \cite[Theorem~1.1]{LRT10}. 

A complementary approximation ansatz combines Gaussians of fixed width with time-varying prefactors, whose evolution resembles the one of the width matrices of the Gaussian envelopes of both the Hagedorn wave packets and the Gaussian beams as well as the evolution of our correction term $\Lambda^t_\eps$. The Herman-Kluk propagator \cite[Theorem~2]{SR09}, \cite[Theorem~1.2]{R10} as well as the frozen Gaussian approximation for strictly hyperbolic systems \cite[Theorem~4.1]{LY12} have been systematically developed to $k$-th order in $\eps$. 

\subsection{Outline}

The following Section~\S\ref{husimi_functions} begins with the comparison of Husimi and Wigner functions and provides  asymptotic expansions in $\eps$ for the transition from Weyl to Anti-Wick quantization and backwards in Lemma~\ref{AW-Weyl_connection} and Lemma~\ref{deconvolution}, respectively. Section~\S\ref{sec:propagation} develops and proves as Theorem~\ref{husimi_egorov} our first result, the second order Egorov approximation for Husimi functions. 
Our second result, Theorem~\ref{corrected_flow-equation}, reformulates the correction term of Theorem~\ref{husimi_egorov}, which has been derived as a time-integral along the $\Phi^t_\eps$-flow, with ordinary differential equations. Section~\S\ref{sec:flow} proposes a symmetric splitting method for the simultaneous discretization of $\Phi^t_\eps$, $\Lambda^t_\eps$, and $\Gamma^t_\eps$. The final 
Section~\S\ref{sec:numerics} presents our numerical results for two molecular Schr\"odinger equations in dimension two and six, respectively.    
\section{Husimi functions and commutators}\label{husimi_functions}

Let $\psi:\R^d\to\C$ be a square integrable function and $\W^\eps(\psi):\R^{2d}\to\R$, 
$$
\W^\eps(\psi)(q,p) = (2\pi\eps)^{-d}\int_{\R^d} e^{i p\cdot y/\eps} \psi(q-\tfrac12 y) \overline\psi(q+\tfrac12 y) dy,
$$
its Wigner function. The Wigner function is a continuous function on phase space. Its marginals are the position and momentum density of $\psi$, respectively,
$$
\int_{\R^d} \W^\eps(\psi)(q,p) dp = |\psi(q)|^2,\qquad \int_{\R^d} \W^\eps(\psi)(q,p) dq = |{\mathcal F}^\eps(\psi)(p)|^2,
$$
where ${\mathcal F}^\eps(\psi)(p)=(2\pi\eps)^{-d/2}\int_{\R^d} e^{-iq\cdot p/\eps}\psi(q)dq$ is the $\eps$-scaled Fourier transform. 
Moreover, associating with a Schwartz function $a:\R^{2d}\to\R$ the Weyl quantized operator
$$
(\opwe(a)\psi)(q) = (2\pi\eps)^{-d} \int_{\R^d}\int_{\R^d} a(\tfrac{q+y}{2},p) e^{i(q-y)\cdot p/\eps}\psi(y) dp dy,
$$
the Wigner function encodes corresponding expectation values via phase space integration,
$$
\langle\psi,\opwe(a)\psi\rangle_{L^2} = \int_{\R^{2d}} a(z) \W^\eps(\psi)(z) dz.
$$ 
However, despite its positive marginals, the Wigner function may attain negative values. This is easily seen for odd functions, 
which must satisfy $$\W^\eps(\psi)(0,0)=-(2\pi\eps)^{-d}\|\psi\|_{L^2}^2~.$$

\subsection{Husimi functions}
A proper smoothing of the Wigner function results in a nonnonegative phase space function, the so-called Husimi function:

\begin{mdef} 
Let $\psi:\R^d\to\C$ be square integrable and $\W^\eps(\psi)$ its Wigner function. Let $a:\R^{2d}\to\R$ be a Schwartz function.  
We denote by 
$$
G_{\eps/2}(z) = (\pi \eps)^{-d}e^{-|z|^2/\eps},\qquad z\in\R^{2d},
$$ 
the centered phase space Gaussian with covariance $\tfrac\eps2{\rm Id}$. Then, 
$$
\H^\eps(\psi) = \W^\eps(\psi)* G_{\eps/2},\qquad \opaw(a) = \opwe(a* G_{\eps/2}) 
$$
are called the {\em Husimi function} of $\psi$ and the {\em Anti-Wick operator} of $a$, respectively.
\end{mdef}

By construction the Husimi function is a smooth function satisfying
\begin{equation}
\label{husimi-equality}
\left\langle \psi, \opaw(a) \psi \right\rangle_{L^2} = \int_{\R^{2d}} a(z) \H^\eps(\psi)(z)\;dz. 
\end{equation}
Especially, $\|\H^\eps(\psi)\|_{L^1}=\|\psi\|_{L^2}^2$. One can prove, that $\H^\eps(\psi)$ is the modulus squared of the FBI transform of $\psi$ and therefore nonnegative.
Moreover, smoothing the Wigner function with Gaussians of smaller covariance does not yield positivity. 
In the following Proposition~\ref{prop:husimi}, we summarize these observations combining arguments of the proofs of \cite[Proposition~2.20]{Z08} and \cite[Theorem~4.2]{B67}. 

\begin{prop}\label{prop:husimi} 
If $\sigma\ge\eps$, then $\W^\eps(\psi)*G_{\sigma/2}\ge0$ for all $\psi\in L^2(\R^d)$. If $\sigma<\eps$, then there 
exists $\psi\in L^2(\R^d)$ such that $\W^\eps(\psi)*G_{\sigma/2}$ attains negative values.
\end{prop}

\begin{proof}
We write
\begin{eqnarray*}
\lefteqn{\left(\W^\eps(\psi)*G_{\sigma/2}\right)(q,p)}\\
&=& (\pi\sigma)^{-d} (2\pi\eps)^{-d} \int_{\R^{3d}} e^{i\xi\cdot y/\eps} \psi(x-\tfrac12 y) \overline\psi(x+\tfrac12y) e^{-(|x-q|^2+|\xi-p|^2)/\sigma} dy dx d\xi\\
&=& (\pi\sigma)^{-d/2} (2\pi\eps)^{-d}
\int_{\R^{2d}} e^{ip\cdot y/\eps} \psi(x-\tfrac12 y)\overline\psi(x+\tfrac12y) e^{-|x-q|^2/\sigma} e^{-\frac{\sigma}{4\eps^2}|y|^2} dy dx,
\end{eqnarray*}
since
$$
\int_{\R^d} e^{i\xi\cdot y/\eps} e^{-|\xi-p|^2/\sigma} d\xi = (\sigma\pi)^{d/2} e^{ip\cdot y/\eps} e^{-\frac{\sigma}{4\eps^2}|y|^2}.
$$
We set $\alpha=\left(\sigma^2/\eps^2-1\right)/(4\sigma)$.  
The change of variables $x-\frac12y=v$, $x+\frac12y=w$ together with the relation
$$
|x-q|^2 + \tfrac14|y|^2 = \tfrac12|v-q|^2 + \tfrac12|w-q|^2
$$
and
$$
\tfrac1\sigma|x-q|^2 + \tfrac{\sigma}{4\eps^2}|y|^2 = \tfrac{1}{2\sigma}|v-q|^2 + \tfrac{1}{2\sigma}|w-q|^2 + \alpha|w-v|^2
$$
yields
$$
\left(\W^\eps(\psi)*G_{\sigma/2}\right)(q,p) = \int_{\R^{2d}} f(v) \overline f(w) e^{-\alpha |w-v|^2} dv dw 
$$
with 
$$
f(v) = (\pi\sigma)^{-d/4} (2\pi\eps)^{-d/2} \; e^{-ip\cdot v/\eps} \psi(v) e^{-|v-q|^2/(2\sigma)}.
$$
The Taylor series of $\alpha\mapsto e^{2\alpha v\cdot w}$ then gives 
$$
\left(\W^\eps(\psi)*G_{\sigma/2}\right)(q,p) = \sum_{k=0}^\infty \frac{(2\alpha)^k}{k!} \sum_{i_1,\ldots,i_k} \left| \int_{\R^{d}} v_{i_1}\cdots v_{i_k} f(v)e^{-\alpha |v|^2} dv \right|^2.
$$
If $\sigma=\eps$, then $\alpha=0$ and 
$$
\forall\psi\in L^2(\R^d):\quad \W^\eps(\psi)*G_{\eps/2} = \left| \int_{\R^{d}} f(v) dv\right|^2\ge0.
$$
If $\sigma>\eps$, then $G_{\sigma/2}=G_{\eps/2}*G_{(\sigma-\eps)/2}$ and 
$$
\forall\psi\in L^2(\R^d):\quad\W^\eps(\psi)*G_{\sigma/2}=\left(\W^\eps(\psi)*G_{\eps/2}\right)*G_{(\sigma-\eps)/2}>0.
$$ 
If $\sigma<\eps$, then $\alpha<0$. Therefore, choosing $\psi(v)=v_1 e^{-|v|^2}$, we obtain 
$$
\left(\W^\eps(\psi)*G_{\sigma/2}\right)(0,0) = \sum_{k\;{\rm odd}}^\infty  \frac{(2\alpha)^k}{k!} \sum_{i_1,\ldots,i_k} \left| \int_{\R^{d}} v_{i_1}\cdots v_{i_k} f(v)e^{-\alpha |v|^2} dv \right|^2 < 0.
$$
\end{proof}

\subsection{Weyl and Anti-Wick quantization}

The nonnegativity of the Husimi function is shared by the Anti-Wick quantization, which satisfies $a\ge0\Rightarrow\opaw(a)\ge0$ by identity (\ref{husimi-equality}).    
The following Lemma expresses $\opaw(a)$ as a Weyl quantized operator, whose 
symbol is obtained from $a$ by adding higher order derivatives. It extends the first order approximation of \cite[Proposition~2.4.3]{L10} to higher orders. 

\begin{lem}\label{AW-Weyl_connection}
Let $a:\R^{2d}\to\R$ be a Schwartz function, $n\in\N$, and $\eps>0$. There exists a family of 
Schwartz functions $r_n^\eps(a):\R^{2d}\to\R$ with $\sup_{\eps>0}\|\opwe(r_n^\eps(a))\|_{\mathcal{L}(L^2)}<\infty$ 
such that     
$$
\opaw(a) = \opwe \left(a + \sum\limits_{k=1}^{n-1}
\frac{\eps^k }{4^k k!}  \Delta^ka \right) + \eps^{n}\opwe(r_n^\eps(a)).
$$
\end{lem}

\begin{proof} We work with
$$
(a*G_{\eps/2})(z) = (\pi \eps)^{-d}\int_{\R^{2d}}a(\zeta)e^{-|\zeta-z|^2/\eps}~d\zeta
$$
and Taylor expand $a$ around the base point $z$, 
\begin{eqnarray*}
a(\zeta) &=& \sum\limits_{|\alpha|=0}^{2n-1}\frac{1}{\alpha!}(\partial^\alpha a)(z)(\zeta-z)^\alpha\\
&& + 2n \sum_{|\alpha|=2n} \frac{(\zeta-z)^\alpha}{\alpha!}\int_0^1(1-\theta)^{2n-1}(\partial^\alpha a)
(z+\theta(\zeta-z))~d\theta
\end{eqnarray*}
with standard multiindex notation.
The terms with $a$-derivatives of odd degree do not contribute, since
$$
\int_{\R^{2d}} f(\zeta-z)e^{-|\zeta-z|^2/\eps}~d\zeta = 0
$$
for odd functions $f$. For the derivatives of even degree we obtain
\begin{eqnarray*}
\lefteqn{(\pi \eps)^{-d}\sum_{|\alpha|=m} \int_{\R^{2d}} \frac{1}{(2\alpha)!} 
(\partial^{2\alpha}a)(z)(\zeta-z)^{2\alpha}e^{-|\zeta-z|^2/\eps}~d\zeta} \\
&=& \sum_{|\alpha|=m} \eps^m \frac{\pi^{-d}}{(2\alpha)!} (\partial^{2\alpha} a)(z) 
\int_{\R^{2d}} y_{i_1}^{2\alpha_1}\cdots y_{i_{2k}}^{2\alpha_{2d}}e^{-|y|^2}~dy.
\end{eqnarray*}
Since
\[
\frac{1}{\sqrt \pi}\int_\R y^{2m} e^{-y^2}~dy = \frac{(2m-1)!!}{2^m} = \frac{(2m)!}{4^m  m! }
\]
with  $(2m-1)!! := (2m-1)\cdot(2m-3)\cdots3\cdot 1$,
we obtain
\begin{eqnarray*}
\lefteqn{\sum_{|\alpha|=m}\eps^{m} \frac{ \pi^{-d} }{(2\alpha)!}  (\partial^{2\alpha} a)(z)
\int_{\R^{2d}} y_{i_1}^{2\alpha_1}\cdots y_{i_{2d}}^{2\alpha_{2d}}e^{-|y|^2}~dy}\\
&=& \frac{\eps^{m}}{4^m} \sum_{|\alpha|=m} \frac{1}{\alpha!} (\partial^{2\alpha} a)(z) ~=~ 
\frac{\eps^{m}}{4^m m!} \Delta^ma(z).
\end{eqnarray*}
Consequently, 
$$
(a*G_{\eps/2})(z) = \sum\limits_{k=0}^{n-1}\frac{\eps^k}{4^k k!}\Delta^k a(z) + \eps^n r_n^\eps(a)(z),
$$
with 
$$
r_n^\eps(a)(z)=\sum_{|\alpha|=2n}\frac{\pi^{-d} 2n}{\alpha!} \int_{\R^{2d}}\int_0^1(1-\theta)^{2n-1}
(\partial^{\alpha}a)(z+\sqrt\eps \theta y)y^\alpha e^{-|y|^2}~d\theta~dy.
$$
Moreover, 
$$
\|\opwe(r_n^\eps(a))\|_{\mathcal{L}(L^2)} < C\sup_{|\alpha|,|\beta|\le\lceil d/2\rceil+1}
\|\partial_q^\alpha\partial_p^\beta a^{(2n)}\|_\infty
$$ 
for some constant $C=C(d,n)>0$ by the Calderon-Vaillancourt Theorem. 
\end{proof}

By means of Lemma~\ref{AW-Weyl_connection}, we can also write $\opwe(a)$ as an Anti-Wick 
quantized operator by additively correcting $a$ with its derivatives.  

\begin{lem}\label{deconvolution}
Let $a:\R^{2d}\to\R$ be a Schwartz function, $n\in\N$, and $\eps>0$. There exists a family of Schwartz functions $\rho_n^\eps(a):\R^{2d}\to\R$ with $\sup_{\eps>0}\|\opwe(\rho_n^\eps(a))\|_{\mathcal{L}(L^2)}<\infty$ 
such that
$$
\opwe(a) = \opaw\left(a+ \sum_{k=1}^{n-1} \frac{(-\eps)^k}{4^k k!}  \Delta^k  a \right)+ 
\eps^n \opwe(\rho_n^\eps(a)).
$$
\end{lem}

\begin{proof} 
Applying Lemma~\ref{AW-Weyl_connection} yields
\begin{eqnarray*}
G_{\eps/2}* \left(a+ \sum_{k=1}^{n-1} \frac{(-\eps)^k}{4^k \cdot k!}  \Delta^k  a \right)
&=& \sum\limits_{j=0}^{n-1} \left( \sum_{k=0}^{n-1-j} 
\frac{\eps^{k+j} (-1)^j}{4^{k+j} \cdot k! \cdot j!} \Delta^{k+j} a\right)~-~\eps^n \rho_n^\eps(a) \\
&=& \sum_{m=0}^{n-1} \frac{\eps^m}{4^m} \Delta^m a \frac{1}{m!} 
\sum_{j=0}^m \binom{m}{j} (-1)^j ~-~\eps^n \rho_n^\eps(a) \\
&=& a ~-~\eps^n \rho_n^\eps(a) 
\end{eqnarray*}
where  the Schwartz function $\rho_n^\eps(a)$ satisfies
\[
\|\opwe(\rho_n^\eps(a))\|_{\mathcal{L}(L^2)}
<C\sup_{|\alpha|,|\beta|\le\lceil d/2\rceil+1} \|\partial_q^\alpha\partial_p^\beta a^{(2n)}\|_\infty
\] 
for some constant $C=C(d,n)>0$. 
\end{proof}

For further reference, we summarize the explicit form of the second order approximations of Lemma~\ref{AW-Weyl_connection} and Lemma~\ref{deconvolution}:
\begin{eqnarray}
\label{AW_second}
\opaw(a) &=& \opwe\left(a + \tfrac{\eps}{4}\Delta a\right) + \eps^{2}\opwe(r_2^\eps(a)),\\
\label{We_second}
\opwe(a) &=& \opaw\left(a - \tfrac{\eps}{4}\Delta a\right)+ \eps^{2}\opwe(\rho_2^\eps(a)). 
\end{eqnarray}

\subsection{Commutators}

When approximating Husimi functions and expectation values for Anti-Wick quantized operators with respect to the solution of semiclassical Schr\"odinger equations, we will face commutators of Weyl and Anti-Wick quantized operators. For their asymptotic expansion we need additional notation: Let $a,b\in{\mathcal S}(\R^{2d})$ be Schwartz functions. Then, the composition of their Weyl quantized operators is the Weyl quantized operator of the Moyal 
product $a\sharp b$, 
\begin{equation}\label{moyal_prod}
\opwe(a)\opwe(b) = \opwe(a\sharp b).
\end{equation}
The commutator $[\opwe(a),\opwe(b)]=\opwe(a)\opwe(b)-\opwe(b)\opwe(a)$ is the Weyl quantization of the Moyal commutator $a\sharp b-b\sharp a$, which has an $\eps$-expansion 
\begin{equation}
\label{eq:sharp}
a \sharp b-b\sharp a = \sum_{k=0}^{n-1} \eps^k \Theta_k(a,b) + \eps^{n} R^\eps_n(a,b),
\end{equation}
where 
\begin{equation}
\label{moyal_expansion}
\Theta_k(a,b)(z) = \frac{i^k}{k!}  \left.A(D)^k (a(z)b(z')-b(z)a(z'))\right|_{z'=z}
\end{equation}
is obtained from the $k$-fold application of 
\begin{equation}\label{bidiff_moyal}
A(D)(a(z)b(z'))= \tfrac{1}{2}J\nabla a(z)\cdot\nabla b(z'),\qquad J = \begin{pmatrix} 0 & {\rm Id}\\ -{\rm Id} & 0\end{pmatrix}
\end{equation}
and $\sup_{\eps>0}\|\opwe(R^\eps_n(a,b))\|_{\mathcal{L}(L^2)}<\infty$, see e.g. \cite[Section 4.3]{Z12}. We note, that $\Theta_k(a,b)=0$ for $k$ even. For commutators of Weyl and Anti-Wick quantized operators we derive the following expansion.

\begin{lem}\label{lem:com}
Let $a,b:\R^{2d}\to\R$ be a smooth function of subquadratic growth and a Schwartz function, respectively, and
let $n\in\N$, and $\eps>0$. Then, $\frac{i}{\eps} \left[ \opwe(a),\opaw(b)\right]$ is essentially self-adjoint in $L^2(\R^d)$ with core ${\mathcal S}(\R^{d})$, and there exists a family of Schwartz functions $\varrho_n^\eps(a,b):\R^{2d}\to\R$ with 
$\sup_{\eps>0}\|\opwe(\varrho^\eps_n(a,b))\|_{\mathcal{L}(L^2)}<\infty$ such that  
$$
\frac{i}{\eps}\left[\opwe(a),\opaw(b)\right] = \sum_{k=0}^{n-1} \eps^k \opwe(\theta_k(a,b)) +
 \eps^n\opwe(\varrho^\eps_n(a,b))
$$
with 
$$
\theta_k(a,b)=i\sum_{\ell+m=k+1} \frac{1}{4^\ell  \ell!} \Theta_m(a,\Delta^\ell b),
$$ 
where $\Theta_m$ has been defined in (\ref{moyal_expansion}).
\end{lem}

\begin{proof}
Since $b$ is a Schwartz function, $\opaw(b)$ maps $L^2(\R^d)$ into $\mathcal{S}(\R^d)$ and hence the commutator with
$\opwe(a)$ is densely defined and essentially selfadjoint.
By Lemma~\ref{AW-Weyl_connection}, 
\begin{eqnarray*}
\lefteqn{\frac{i}{\eps}\left[\opwe(a),\opaw(b)\right]}\\
&=&
\frac{i}{\eps}\sum_{k=0}^{n}\frac{\eps^k}{4^k k!}
\left[\opwe(a),\opwe(\Delta^k b)\right]
+ i\,\eps^n\left[\opwe(a),\opwe(r^\eps_{n+1}(b))\right].
\end{eqnarray*}
Using the Moyal expansion~(\ref{eq:sharp}), we then obtain
\begin{eqnarray*}
\lefteqn{\frac{i}{\eps}\left[\opwe(a),\opaw(b)\right]}\\ 
&=&
\frac{i}{\eps}\sum_{k=0}^n \sum_{m=0}^{n-k} \frac{\eps^{k+m}}{4^k k!}
\opwe\left(\Theta_m(a,\Delta^k b)\right)\\
&+&i\,\eps^n \left(\sum_{k=0}^{n}\frac{1}{4^k k!}\opwe\left(R^\eps_{n-k+1}(a,\Delta^k b)\right)
+\opwe\left( a\sharp r^\eps_{n+1}(b)-r^\eps_{n+1}(b)\sharp a\right)\right)\\
&=:&
\frac{i}{\eps}\sum_{k=0}^n \sum_{m=0}^{n-k} \eps^{k+m}\frac{\eps^{k+m}}{4^k k!}
\opwe\left(\Theta_m(a,\Delta^k b)\right) + \eps^n\opwe(\varrho^\eps_n(a,b)).
\end{eqnarray*}
with some family of Schwartz functions $(\varrho^\eps_n(a,b))_{\eps>0}$.
Finally, we use the fact that $\Theta_{0}(a,b)=0$ to rewrite
\begin{eqnarray*}
\frac{i}{\eps}\sum_{k=0}^n \sum_{m=0}^{n-k} \frac{\eps^{k+m}}{4^k k!} \Theta_m(a,\Delta^k b)
&=&
\frac{i}{\eps}\sum_{p=0}^{n} \eps^p  \sum_{k+m=p} \frac{1}{4^k k!} \Theta_m(a,\Delta^k b)\\
&=&
i\sum_{p=0}^{n-1} \eps^p  \sum_{k+m=p+1} \frac{1}{4^k k!} \Theta_m(a,\Delta^k b).  
\end{eqnarray*}
\end{proof}

\begin{rem}
We note that $\theta_0(a,b)=\{a,b\}$, where $\{a,b\}=J\nabla a\cdot \nabla b$ denotes the Poisson bracket. Since $\theta_1(a,b)=\frac14 \{a,\Delta b\}$, we have 
$$
\frac{i}{\eps}\left[\opwe(a),\opaw(b)\right] = \opwe(\{a,b+\tfrac{\eps}{4}\Delta b\}) + \eps^2\opwe(\varrho^\eps_2(a,b)).
$$
\end{rem}

%

\section{Propagation}\label{sec:propagation}

In this section we apply the systematic transition between Weyl- and Anti-Wick calculus to prove a second order Egorov approximation for Anti-Wick quantized symbols. We can allow for unitary time evolutions 
$e^{-iHt/\eps}$ whose Hamiltonian $H=\opwe(h)$ is obtained from a smooth symbol $h:\R^{2d}\to\R$ of subquadratic growth. That is, for all $|\gamma|\ge2$ there exists $C_\gamma>0$ such that  
$$
\|\partial^\gamma h\|_\infty < C_\gamma.
$$ 

In this situation, $H$ is essentially self-adjoint in $L^2(\R^d)$ with core ${\mathcal S}(\R^{d})$, $e^{-iHt/\eps}$ is a well-defined unitary operator on $L^2(\R^d)$, and the Hamiltonian flow $\Phi^t:\R^{2d}\to\R^{2d}$ associated with $h$ is a smooth mapping for all $t\in\R$. Schr\"odinger operators 
$$
-\tfrac{\eps^2}{2}\Delta+V
$$ 
with smooth subquadratic potential $V$ are one important example. 

\subsection{A first Egorov type approximation}

The proof of the Egorov type approximation is mainly built on the following commutator estimate: 
\begin{lem}\label{commutator}
Let $\eps>0$. Let $b,c:\R^{2d}\to\R$ be a smooth function of subquadratic growth and a Schwartz function, respectively. 
Then, $\frac{i}{\eps} \left[ \opwe(b),\opaw(c)\right]$ is essentially self-adjoint in $L^2(\R^d)$ with core ${\mathcal S}(\R^{d})$, 
and there exists a family of Schwartz functions $\chi_2^\eps(b,c):\R^{2d}\to\R$ with $\sup_{\eps>0}\|\opwe(\chi^\eps_2(b,c))\|_{\mathcal{L}(L^2)}<\infty$ such that
\begin{eqnarray*}
\lefteqn{\frac{i}{\eps} \left[ \opwe(b),\opaw(c)\right]}\\
&=& 
\opaw\left( \{b-\tfrac{\eps}{4} \Delta b,c\} - \tfrac{\eps}{2}\,{\rm tr}(J\, D^2b\, D^2c)\right) +\eps^2 \opwe(\chi^\eps_2(b,c)).
\end{eqnarray*}
\end{lem}

\begin{proof}
Since $c$ is a Schwartz function, $\opaw(c)$ maps $L^2(\R^d)$ into ${\mathcal S}(\R^d)$, and the commutator with $\opwe(b)$ is well-defined. 
By Lemma~\ref{lem:com}, 
$$
\frac{i}{\eps}\left[\opwe(b),\opaw(c)\right]= \opwe(\{b,c+\tfrac{\eps}{4}\Delta c\}) + \eps^2\opwe(\varrho_2^\eps(b,c))
$$
By Lemma~\ref{deconvolution}, 
\begin{eqnarray*}
\opwe(\{b,c+\tfrac{\eps}{4}\Delta c\}) &=& \opaw(\{b,c+\tfrac{\eps}{4}\Delta c\} - \tfrac\eps4\Delta\{b,c\})\\
&&
+\,\eps^2\opwe\left(\rho_2^\eps(\{b,c+\tfrac{\eps}{4}\Delta c\}\right)-\tfrac{\eps^2}{16}\opaw(\Delta\{b,\Delta c\}).
\end{eqnarray*}
By Lemma~\ref{AW-Weyl_connection},
$$
\opaw(\Delta\{b,\Delta c\})=\opwe(\Delta\{b,\Delta c\})+\eps\opwe(r^\eps_1(\Delta\{b,\Delta c\})),
$$
and therefore 
$$
\frac{i}{\eps}\left[\opwe(b),\opaw(c)\right]=\opaw(\{b,c+\tfrac{\eps}{4}\Delta c\} - \tfrac\eps4\Delta\{b,c\})+\eps^2\opwe(\chi^\eps_2(b,c))
$$
with
$$
\chi^\eps_2(b,c)=\varrho_2^\eps(b,c)+\rho_2^\eps(\{b,c+\tfrac{\eps}{4}\Delta c\})-\tfrac{1}{16}\Delta\{b,\Delta c\}-\tfrac{\eps}{16} r^\eps_1(\Delta\{b,\Delta c\}).
$$
a Schwartz function. We write $\Delta\{b,c\}= \{\Delta b,c\} + \{b,\Delta c\} + 2\sum_{k=1}^{2d} \{\partial_{z_k}b,\partial_{z_k}c\}$ and compute 
\begin{eqnarray*}
\sum_{k=1}^{2d} \{\partial_{z_k}b,\partial_{z_k}c\} &=& 
\sum_{k=1}^{2d} J\,D^2b(:,k) \cdot D^2c(:,k) = 
\sum_{k=1}^{2d} D^2 c(k,:) (J\,D^2 b)(:,k)\\
&=& {\rm tr}(J\,D^2 b\,D^2 c),
\end{eqnarray*}
where $M(:,k)$ and $M(k,:)$ denote the $k$th column and the $k$th row of a matrix~$M$, respectively. Therefore,
$$
\{b,c+\tfrac{\eps}{4}\Delta c\} - \tfrac\eps4\Delta\{b,c\}=\{b-\tfrac{\eps}{4} \Delta b,c\} - \tfrac{\eps}{2}\,{\rm tr}(J\, D^2b\, D^2c),
$$
which concludes the proof.
\end{proof}

\begin{rem} 
If $b$ is a polynomial of degree one, then the result of Lemma~\ref{commutator} obviously simplifies to 
$$
\tfrac{i}{\eps} \left[ \opwe(b),\opaw(c)\right]=\opaw\left(\{b,c\}\right).
$$ 
If $b$ is a polynomial of degree two, then the only simplification is due to $\{b-\frac\eps4\Delta b,c\}=\{b,c\}$, since $\chi^\eps_2(b,c)$ is not identical zero.  
\end{rem}

Having prepared the necessary estimates, we state and prove our approximation result for the unitary propagation of Anti-Wick operators. 

\begin{thm}\label{husimi_egorov}
Let $h:\R^{2d}\to\R$ be a smooth function of subquadratic growth. Let $a:\R^{2d}\to\R$ be a Schwartz function and $t\in\R$. There exists a constant $C=C(a,h,t)>0$ such that all $\eps>0$ 
$$
\left\|e^{iHt/\eps}\opaw(a)e^{-iHt/\eps} - 
\opaw(a \circ \Phi^t_\eps -\tfrac{\eps}{2}\, \Xi^t_\eps(a)) \right\|_{\mathcal{L}(L^2)} \le C \eps^2 
$$
with 
\begin{equation}
\label{correction-flow} 
\Xi^t_\eps(a) = \int_0^t {\rm tr}\left(J\,D^2h\, D^2(a\circ\Phi^\tau_\eps) \right) \circ \Phi^{t-\tau}_\eps~d\tau
\end{equation}
and $\Phi^t_\eps:\R^{2d}\to\R^{2d}$ the Hamiltonian flow associated with $h_\eps=h-\frac\eps4 \Delta h$.
\end{thm}

\begin{proof}
We denote 
$$
a(t)=a \circ \Phi^t_\eps -\tfrac{\eps}{2}\, \Xi^t_\eps(a)
$$ 
and observe that $a(t):\R^{2d}\to\R$ is a Schwartz function. Since $a(0)=a$, we have
\begin{eqnarray*}
&&e^{iHt/\eps}\opaw(a)e^{-iHt/\eps} - \opaw(a(t))\\
&&= \int_0^t \frac{d}{ds}\left(e^{iHs/\eps}\opaw(a(t-s))e^{-iHs/\eps}\right) ds\\
&&= \int_0^t e^{iHs/\eps} \left(\frac{i}{\eps} \left[H,\opaw(a(t-s))\right] - \opaw(\partial_ta(t-s))\right) e^{-iHs/\eps}ds.
\end{eqnarray*}
The second order commutator expansion of Lemma~\ref{commutator} yields
\begin{eqnarray*}
\lefteqn{\frac{i}{\eps} \left[H,\opaw(a(t-s))\right]=\frac{i}{\eps} \left[\opwe(h),\opaw(a(t-s))\right]=}\\
&&
\opaw\left( \{h_\eps,a(t-s)\} - \tfrac{\eps}{2}\,{\rm tr}(J\, D^2h\, D^2a(t-s))\right) +\eps^2 \opwe(\rho^\eps_2(h,a(t-s))).
\end{eqnarray*}
Now, we compute
$$
\partial_t(a\circ\Phi^{t-s}_\eps)=\{h_\eps,a\circ\Phi^{t-s}_\eps\}
$$
and 
$$
\partial_t\, \Xi^{t-s}_\eps(a) = {\rm tr}\left(J\,D^2h\, D^2(a\circ\Phi^{t-s}_\eps) \right) + \{h_\eps,\Xi^{t-s}_\eps(a)\}.
$$
Therefore, 
$$
\partial_ta(t-s)=\{h_\eps,a(t-s)\}-\tfrac{\eps}{2}\,{\rm tr}\left(J\,D^2h\, D^2(a\circ\Phi^{t-s}_\eps) \right)
$$
and
\begin{eqnarray*}
\lefteqn{\frac{i}{\eps} \left[H,\opaw(a(t-s))\right] - \opaw(\partial_ta(t-s))=}\\
&& 
\tfrac{\eps^2}{4} \opaw{\rm tr}(J\, D^2h\, D^2 \Xi^{t-s}_\eps(a)) + \eps^2 \opwe(\rho^\eps_2(h,a(t-s))). 
\end{eqnarray*}
The observation that 
$$
{\rm tr}(J\, D^2h\, D^2 \Xi^{t-s}_\eps(a)),\;\rho^\eps_2(h,a(t-s))\in{\mathcal S}(\R^{2d})
$$ 
concludes the proof.
\end{proof}

\subsection{Corrections by ordinary differential equations}

Our next task is to reformulate the time evolution of the correction term 
$$
\Xi_\eps^t(a) = 
\int_0^t {\rm tr}\left(J\, D^2h\, D^2(a\circ\Phi_\eps^\tau) \right) \circ \Phi_\eps^{t-\tau}~d\tau
$$
using ordinary differential equations which are independent of the observable~$a$. 

\begin{lem}\label{lem:lambda_gamma} 
Let $a:\R^{2d}\to\R$ be a Schwartz function and $h:\R^{2d}\to\R$ a smooth function of subquadratic growth. Let $\eps>0$ and $\Phi^t_\eps:\R^{2d}\to\R^{2d}$ the Hamiltonian 
flow associated with $h_\eps=h-\frac\eps4\Delta h$. Then, for all $t\in\R$,
$$
\Xi_\eps^t(a) = 
{\rm tr}\left(\widetilde\Lambda^t_\eps \,(D^2a\circ\Phi^t_\eps)\right)
+\widetilde\Gamma^t_\eps \cdot (\nabla a\circ \Phi^t_\eps) 
$$
with $\widetilde\Lambda^t_\eps:\R^{2d}\to\R^{2d\times 2d}$, 
\begin{equation}\label{eq:lambda}
\widetilde\Lambda^t_\eps= \int_0^t \left( (D\Phi^\tau_\eps)^\T\, J\,D^2h\, D\Phi^\tau_\eps \right)\circ\Phi^{t-\tau}_\eps d\tau,
\end{equation}
and $\widetilde\Gamma^t_\eps:\R^{2d}\to\R^{2d}$
\begin{equation}\label{eq:gamma}
\widetilde\Gamma^t_\eps= \int_0^t \sum_{k,l=1}^{2d} \left((J\,D^2h)_{kl}\,\partial^2_{z_k z_l}\Phi_\eps^\tau \right)\circ\Phi^{t-\tau}_\eps d\tau,
\end{equation}
where the Jacobian of $\Phi^\tau_\eps$ is denoted as $D\Phi^\tau_\eps=(\nabla(\Phi^\tau_\eps)_1,\ldots,\nabla(\Phi^\tau_\eps)_{2d})$. 
\end{lem}

\begin{proof}
We compute
$$
\partial^2_{z_k z_l}(a\circ\Phi_\eps^\tau) = (D^2a \circ \Phi_\eps^\tau)(\partial_{z_l}\Phi_\eps^\tau,\partial_{z_k}\Phi_\eps^\tau) + 
(\nabla a\circ \Phi_\eps^\tau)\cdot \partial^2_{z_k z_l}\Phi_\eps^\tau.
$$
Then,
\begin{eqnarray*}
\lefteqn{{\rm tr}\left(\widetilde\Lambda^t_\eps \,(D^2a\circ\Phi^t_\eps)\right)=\int_0^t {\rm tr}\left( (D\Phi^\tau_\eps)^\T\, J\,D^2h\, D\Phi^\tau_\eps \,(D^2a\circ\Phi^\tau_\eps) \right)\circ\Phi^{t-\tau}_\eps d\tau}\\
&=&
\int_0^t {\rm tr}\left( J\,D^2h\, D\Phi^\tau_\eps (D^2a\circ\Phi^\tau_\eps)\,(D\Phi^\tau_\eps)^\T \right)\circ\Phi^{t-\tau}_\eps d\tau\\
&=&
\sum_{k,l=1}^{2d} \int_0^t \left( (J\,D^2h)_{kl} (D^2a \circ \Phi_\eps^\tau)(\partial_{z_l}\Phi_\eps^\tau,\partial_{z_k}\Phi_\eps^\tau)\right)\circ\Phi^{t-\tau}_\eps d\tau
\end{eqnarray*}
and
\begin{eqnarray*}
\widetilde\Gamma^t_\eps \cdot (\nabla a\circ \Phi^t_\eps)  = 
\int_0^t \sum_{k,l=1}^{2d} \left((J\,D^2h)_{kl}\,\partial^2_{z_k z_l}\Phi_\eps^\tau \cdot (\nabla a\circ \Phi^\tau_\eps) \right)\circ\Phi^{t-\tau}_\eps d\tau.
\end{eqnarray*}
Adding the two terms gives the claimed identity.
\end{proof}

Having written the correction term $\Xi_\eps^t(a)$ as the sum of a trace and an inner product, we can derive the desired $a$-independent ordinary differential equations for the time evolution of $\Xi_\eps^t(a)$. 

\begin{thm}\label{corrected_flow-equation}
Let $a:\R^{2d}\to\R$ be a Schwartz function, $h:\R^{2d}\to\R$ a smooth function of subquadratic growth, and $t\in\R$. Then, there exists a constant $C=C(a,h,t)>0$ such that for all $\eps>0$
$$
\left\|e^{iHt/\eps}\opaw(a)e^{-iHt/\eps} - \opaw(\Psi_\eps^t(a))\right\|_{\mathcal{L}(L^2)} \le C\eps^2
$$
with
$$
\Psi_\eps^t(a) = a\circ\Phi_\eps^t - \tfrac{\eps}{2} \left( {\rm tr}(\Lambda^t_\eps \,(D^2a\circ\Phi^t_\eps)) + \Gamma_\eps^t \cdot(\nabla a\circ \Phi_\eps^t)\right),
$$
where $\Phi^t_\eps:\R^{2d}\to\R^{2d}$ is the Hamiltonian flow of $h_\eps=h-\frac\eps4\Delta h$, 
\begin{equation}\label{eq:ham}
\partial_t\Phi^t_\eps=J\,\nabla h_\eps\circ\Phi^t_\eps,
\end{equation}
and $\Lambda_\eps^t:\R^{2d}\to\R^{2d\times 2d}$, $\Gamma^t_\eps:\R^{2d}\to\R^{2d}$ solve 
\begin{eqnarray}\label{lyapunov}
\partial_t\Lambda_\eps^t &=& M_\eps(t) +  M_\eps(t)\, \Lambda_\eps^t + 
\Lambda_\eps^t\,M_\eps(t)^{\rm T},\qquad\Lambda_\eps^0=0\\*[1.5ex]
\label{gamma-equation}
\partial_t\Gamma_\eps^t &=&  M_\eps(t)\, \Gamma_\eps^t + {\rm tr}(C_i(t)^{\rm T}\, \Lambda_\eps^t)_{i=1}^{2d},\qquad\Gamma_\eps^0= 0
\end{eqnarray}
with 
\begin{eqnarray*}
M_\eps(t):\R^{2d}\to\R^{2d\times 2d},&& M_\eps(t)=J\, D^2h \circ \Phi_\eps^t,\\
C_i(t):\R^{2d}\to\R^{2d\times 2d},&& (C_i(t))_{jk}=\partial_k(J\, D^2 h)_{ij}\circ\Phi_\eps^t.
\end{eqnarray*}
\end{thm}

\begin{proof} We start with the reformulation of $\Xi_\eps^t(a)$ of Lemma~\ref{lem:lambda_gamma} and observe that both relations (\ref{eq:lambda}) and (\ref{eq:gamma}) for $\widetilde\Lambda_\eps^t$ and $\widetilde\Gamma^t_\eps$, respectively, involve time integrals of the form
$$
\int_0^t \left(f_1(\tau)g f_2(\tau)\right)\circ\Phi^{t-\tau}_\eps d\tau. 
$$
We compute
$$
\partial_t \int_0^t \left(f_1(\tau)g f_2(\tau)\right)\circ\Phi^{t-\tau}_\eps d\tau = 
\int_0^t \partial_t \left(\left(f_1(\tau)g f_2(\tau)\right)\circ\Phi^{t-\tau}_\eps\right) d\tau + f_1(t)g f_2(t)
$$
and
\begin{eqnarray*}
\lefteqn{\partial_t\left(\left(f_1(\tau)g f_2(\tau)\right)\circ\Phi^{t-\tau}_\eps\right)}\\
&=&
-\partial_\tau\left(\left(f_1(\tau)g f_2(\tau)\right)\circ\Phi^{t-\tau}_\eps\right) 
+\partial_\tau\left(f_1(\tau)g f_2(\tau)\right)\circ \Phi^{t-\tau}_\eps.
\end{eqnarray*}
We obtain
\begin{eqnarray}\nonumber
\lefteqn{\partial_t \int_0^t \left(f_1(\tau)g f_2(\tau)\right)\circ\Phi^{t-\tau}_\eps d\tau}\\ 
\label{eq:timeder}
&=& 
\left(f_1(0)g f_2(0)\right)\circ\Phi^{t}_\eps
+\int_0^t \partial_\tau\left(f_1(\tau)g f_2(\tau)\right)\circ \Phi^{t-\tau}_\eps d\tau.
\end{eqnarray}
This implies for 
$$
\widetilde\Lambda^t_\eps=\int_0^t \left( (D\Phi^\tau_\eps)^\T\, J\,D^2h\, D\Phi^\tau_\eps \right)\circ\Phi^{t-\tau}_\eps d\tau
$$ 
that
\begin{eqnarray*}
\partial_t \widetilde\Lambda^t_\eps &=& M_\eps(t)  + \int_0^t \left( (J\,D^2 h_\eps\circ\Phi^\tau_\eps)(D\Phi^\tau_\eps)^\T\, J\,D^2h\, D\Phi^\tau_\eps\right)\circ\Phi^{t-\tau}_\eps d\tau\\  
&&+\int_0^t \left((D\Phi^\tau_\eps)^\T\, J\,D^2h\, D\Phi^\tau_\eps (J\,D^2 h_\eps\circ\Phi^\tau_\eps)^T\right)\circ\Phi^{t-\tau}_\eps d\tau,\\
&=&
M_\eps(t) + (J\,D^2 h_\eps\circ\Phi^t_\eps)\widetilde\Lambda^t_\eps + 
\widetilde\Lambda_\eps^t (J\,D^2 h_\eps\circ\Phi^t_\eps)^\T,
\end{eqnarray*}
since the Hamilton equation (\ref{eq:ham}) implies $\partial_\tau\partial_{z_k}\Phi^\tau_\eps = \left(J\,D^2 h_\eps\circ\Phi^\tau_\eps\right) \partial_{z_k}\Phi^\tau_\eps$ and 
\begin{eqnarray*}
\partial_\tau (D\Phi^\tau_\eps)^\T &=& (J\,D^2 h_\eps\circ\Phi^\tau_\eps) (D\Phi^\tau_\eps)^\T,\\
\partial_\tau D\Phi^\tau_\eps &=& D\Phi^\tau_\eps (J\,D^2 h_\eps\circ\Phi^\tau_\eps)^\T.
\end{eqnarray*}
Since $J\,D^2 h_\eps\circ\Phi^t_\eps=M_\eps(t)+O(\eps)$, the solution $\Lambda^t_\eps$ of equation~(\ref{lyapunov}) satisfies $\Lambda^t_\eps=\widetilde\Lambda^t_\eps + O(\eps)$, and we obtain
$$
a\circ\Phi^t_\eps - \tfrac\eps2\,\Xi^t_\eps(a) = 
a\circ\Phi^t_\eps - \tfrac\eps2\left( {\rm tr}(\Lambda^t_\eps \,(D^2a\circ\Phi^t_\eps))
+\widetilde\Gamma^t_\eps \cdot (\nabla a\circ \Phi^t_\eps) \right) + O(\eps^2).
$$ 
For the time derivative of 
$$
\widetilde\Gamma^t_\eps= \int_0^t \sum_{k,l=1}^{2d} \left((J\,D^2h)_{kl}\,\partial^2_{z_k z_l}\Phi_\eps^\tau \right)\circ\Phi^{t-\tau}_\eps d\tau
$$ 
equation~(\ref{eq:timeder}) implies
\begin{eqnarray*}
\partial_t\widetilde\Gamma^t_\eps = \int_0^t \sum_{k,l=1}^{2d} \left((J\,D^2h)_{kl}\,\partial_\tau\partial^2_{z_k z_l}\Phi_\eps^\tau \right)\circ\Phi^{t-\tau}_\eps d\tau.
\end{eqnarray*}
We consider $\partial_{z_k}(\Phi^\tau_\eps)_i = \left(\left(J\,D^2 h_\eps\circ\Phi^\tau_\eps\right) \partial_{z_k}\Phi^\tau_\eps\right)_i$ and compute 
\begin{eqnarray*}
\lefteqn{\partial_\tau \partial^2_{z_k z_l}(\Phi^\tau_\eps)_i}\\ 
&=& 
\sum_{j,n=1}^{2d} \left(\partial_n(J\, D^2 h_\eps)_{ij}\circ \Phi^\tau_\eps\right)\,\partial_{z_l}(\Phi^\tau_\eps)_n\, \partial_{z_k}(\Phi^\tau_\eps)_j
+ \left((J\,D^2 h_\eps\circ\Phi^\tau_\eps)\partial^2_{z_kz_l}\Phi^\tau_\eps\right)_i\\
&=:&
\sum_{j,n=1}^{2d} (\widetilde C_i(\tau))_{jn}\,(D\Phi^\tau_\eps)_{ln}\,(D\Phi^\tau_\eps)_{kj}
+ \left((J\,D^2 h_\eps\circ\Phi^\tau_\eps)\partial^2_{z_kz_l}\Phi^\tau_\eps\right)_i.
\end{eqnarray*}
Therefore, 
\begin{eqnarray*}
\partial_t(\widetilde\Gamma^t_\eps)_i &=& ((J\,D^2 h_\eps\circ\Phi^t_\eps)\,\widetilde\Gamma^t_\eps)_i\\ 
&&+\; 
\sum_{j,k,l,n=1}^{2d} (\widetilde C_i(t))_{jn} 
\int_0^t \left( (J\,D^2h)_{kl}\,(D\Phi^\tau_\eps)_{ln}\, (D\Phi^\tau_\eps)_{kj}\right)
\circ\Phi^{t-\tau}_\eps d\tau\\
&=&
((J\,D^2 h_\eps\circ\Phi^t_\eps)\,\widetilde\Gamma^t_\eps)_i + 
\sum_{j,n=1}^{2d} (\widetilde C_i(t))_{jn} (\widetilde \Lambda^t_\eps)_{jn}\\
&=&
((J\,D^2 h_\eps\circ\Phi^t_\eps)\,\widetilde\Gamma^t_\eps)_i + {\rm tr}(\widetilde C_i(t)^{\rm T}\,\widetilde\Lambda_\eps^t).
\end{eqnarray*}
Since $J\,D^2 h_\eps\circ\Phi^t_\eps=M^\eps(t)+O(\eps)$, $C_i(t)=\widetilde C_i(t)+O(\eps)$, and $\Lambda^t_\eps=\widetilde\Lambda^t_\eps+O(\eps)$, the solution $\Gamma^t_\eps$ of equation~(\ref{gamma-equation}) satisfies $\Gamma^t_\eps=\widetilde\Gamma^t_\eps+O(\eps)$, and we conclude 
$$
a\circ\Phi^t_\eps - \tfrac\eps2\,\Xi^t_\eps(a) = 
a\circ\Phi^t_\eps - \tfrac\eps2\left( {\rm tr}(\Lambda^t_\eps \,(D^2a\circ\Phi^t_\eps))
+\Gamma^t_\eps \cdot (\nabla a\circ \Phi^t_\eps) \right) + O(\eps^2).
$$
\end{proof}

\begin{cor}\label{cor}
Under the assumptions of Theorem~\ref{corrected_flow-equation}, there exists a constant $C=C(a,h,t)>0$ such that $\psi_t=e^{-iHt/\eps}\psi_0$ satisfies
\begin{equation}\label{eq:approx_husimi}
\left|\langle\psi_t,\opwe(a)\psi_t\rangle_{L^2} - \int_{\R^{2d}} F^t_\eps(a_\eps)(z)\, \H^\eps(\psi_0)(z)\, dz \right| \le C \eps^2 
\end{equation}
for all $\psi_0\in L^2(\R^d)$ with $\|\psi_0\|_{L^2}=1$, where 
$$
F_\eps^t(a) = a_\eps\circ\Phi_\eps^t - \tfrac{\eps}{2} \left( {\rm tr}(\Lambda^t_\eps \,(D^2a\circ\Phi^t_\eps)) + \Gamma_\eps^t \cdot(\nabla a\circ \Phi_\eps^t)\right).
$$
\end{cor}

\begin{proof} We have
\begin{eqnarray*}
\lefteqn{\langle\psi_t,\opwe(a)\psi_t\rangle_{L^2} = \langle\psi_t,\opaw(a_\eps)\psi_t\rangle_{L^2} + O(\eps^2)}\\
&=& \langle\psi_0,\opaw(\Psi^t_\eps(a_\eps))\psi_0\rangle_{L^2} + O(\eps^2)
= \int_{\R^{2d}} \Psi^t_\eps(a_\eps)(z)\, \H^\eps(\psi_0)(z)\, dz + O(\eps^2)
\end{eqnarray*}
and $\Psi_\eps^t(a_\eps) = F_\eps^t(a) + O(\eps^2)$.
\end{proof}

\section{Discretization}\label{sec:flow}

For a numerical realization of the semiclassical approximation (\ref{eq:approx_husimi}), we proceed in two steps. First, we discretize the phase space integral by equiweighted quadrature with nodes $z_1,\ldots,z_N\in\R^{2d}$ distributed according to the Husimi function $\H^\eps(\psi_0)$. Then, we discretize the time evolution of $\Phi^t_\eps$, $\Lambda^t_\eps$, and $\Gamma^t_\eps$ initialized with the sampling points $z_1,\ldots,z_N$. For both discretization steps, we compute the $\eps$-correction with lower accuracy.

\subsection{Numerical quadrature}
To discretize the phase space integral  
\begin{eqnarray*}
\lefteqn{\int_{\R^{2d}} F^t_\eps(a_\eps)(z)\, \H^\eps(\psi_0)(z)\, dz}\\
&=&
\int_{\R^{2d}} (a_\eps\circ\Phi^t_\eps)(z)\, \H^\eps(\psi_0)(z)\, dz - \tfrac\eps2 \int_{\R^{2d}} \Xi^t_\eps(a)(z)\, \H^\eps(\psi_0)(z)\, dz\\*[1ex]
&=:&
I(a_\eps\circ\Phi^t_\eps) - \tfrac\eps2 I(\Xi^t_\eps(a))
\end{eqnarray*}
we use the fact, that the Husimi function of a normalized wave function $\psi_0$ is a probability density, and work either with Markov chain Monte Carlo quadrature or with Quasi-Monte Carlo quadrature. 
That is, we approximate 
$$
I(f)\approx
\frac1N\sum_{j=1}^N f(z_j) =: Q_N(f)
$$
with quadrature nodes $z_1,\ldots,z_N\in\R^{2d}$ distributed according to $\H^\eps(\psi_0)$. Thanks to the small $\eps$-prefactor, the correction integral $I(\Xi^t_\eps(a))$ can be computed with less quadrature nodes than the leading order $I(a_\eps\circ\Phi^t_\eps)$.

\subsubsection{Markov chain Monte Carlo quadrature}
We use Metropolis Monte Carlo and generate sample points $z_1,\ldots,z_N$, which form a Markov chain with stationary distribution $\H^\eps(\psi_0)$, see e.g. \cite[\S 4.1]{KLW09}. If the chain is uniformly ergodic, then a central limit theorem holds, and there exists a constant $\gamma=\gamma(f)>0$ such that for all $c>0$
$$
\lim_{N\to\infty} \P\left(
\left| I(f) -  Q_N(f)\right| \le \frac{c\gamma}{\sqrt{N}}
\right) = 
\frac{1}{\sqrt{2\pi}}\int_{-c}^c e^{-\tau^2/2} d\tau.
$$ 
We note, that for the numerical experiments in Section \S\ref{sec:numerics}, the constant $\gamma(f)$ shows a beneficial dependence on the semiclassical parameter $\eps>0$, since the variance of the integrands decreases with decreasing $\eps$. 

\subsubsection{Quasi-Monte Carlo quadrature}\label{QMC-theory}
For Gaussian wave packets $\psi_0$, the Husimi function $\H^\eps(\psi_0)$ is a phase space Gaussian, and we view it as the density of a multivariate normal distribution. For notational simplicity, we focus on the special case of an isotropic Gaussian wave packet $g_{z_0}$ centered in $z_0=(q_0,p_0)\in\R^{2d}$, 
\begin{equation}
\label{eq:gauss}
g_{z_0}(q)=(\pi \eps)^{-d/4}\exp\left(-\tfrac{1}{2\eps}|q-q_0|^2 + \tfrac{i}{\eps} p_0\cdot (q-q_0)\right).
\end{equation}
In this case, the Husimi function is the isotropic phase space Gaussian
$$ 
\H^\eps(g_{z_0})(z) = (2\pi\eps)^{-d} \exp\left( -\tfrac{1}{2\eps}|z-z_0|^2\right),
$$
that is, the density of a normal distribution with mean $z_0$ and covariance~$\eps\Id$. 
Generating points of low star discrepancy with respect to the uniform distribution on $[0,1]^{2d}$, 
that is, for example Sobol points, we use the cumulative distributive function of the normal distribution 
and map them to points $z_1,\ldots,z_N$ such that their star discrepancy with respect to the normal distribution  
\begin{eqnarray*}
\lefteqn{{\mathcal D}^*(z_1,\ldots,z_N)}\\
&=& \sup_{\alpha\in\R^{2d}} \left| \tfrac{1}{N} \#\{z_j: z_j\in(-\infty,\alpha), j=1,\ldots,N\} - \H^\eps(\psi_0)((-\infty,\alpha))\right|
\end{eqnarray*}
is optimal, that is, ${\mathcal D}^*(z_1,\ldots,z_N)=O((\log N)^{2d}/N)$. Then, the Koksma-Hlawka inequality yields a constant $\gamma=\gamma(f)>0$ such that 
$$
\left|  I(f)- Q_N(f)\right| \le \frac{\gamma (\log N)^{c_{d}}}{N},
$$
where $c_{d}\ge 2d$ dependends on the dimension of phase space, see \cite[3.2]{LR10}. We note, that as for the Monte Carlo quadrature, the numerical experiments of Section \S\ref{sec:numerics} show a favourable dependance of the constant $\gamma(f)$ on the semiclassical parameter~$\eps>0$.

\subsection{A time splitting scheme}

For numerically computing $\Phi^t_\eps$, $\Lambda^t_\eps$, and $\Gamma^t_\eps$, we discretize the Hamiltonian flow~(\ref{eq:ham}) with a higher order symplectic scheme for an 
accurate approximation of the leading order contribution $a_\eps\circ\Phi^t_\eps$. The correction equations 
(\ref{lyapunov}) and (\ref{gamma-equation}) are simultaneously discretized after vectorizing the Lyapunov matrix differential equation (\ref{lyapunov}). For this purpose we use the Kronecker product
$$
A\otimes B = \left(\begin{array}{clcc} a_{11}B&\cdots&a_{1n}B\\
\vdots& &\vdots \\ a_{n1}B&\cdots&a_{nn}B\end{array} \right)~~\in\R^{n\cdot m \times n\cdot m}~~.
$$
of two matrices $A\in\R^{n\times n}$ and $B\in\R^{m\times m}$ as well as the rowwise vectorization 
$$
{\rm Vec}(A)=\left(a_{11},a_{12},\hdots,a_{nn}\right)^\T =\vec A\in\R^{n^2}.
$$

\begin{lem} The differential equations (\ref{eq:ham}), (\ref{lyapunov}), and (\ref{gamma-equation}) are equivalent to the coupled vector differential equation with $4d+4d^2$ components 
\begin{equation}
\label{complete_variational}
\partial_t \begin{pmatrix}\Phi^t_\eps \\ \vec\Lambda_\eps^t\\ \Gamma_\eps^t \end{pmatrix} =
\begin{pmatrix}\Id_{2d}&0&0 \\ 0& K_\eps(t)&0\\ 0 & C_\eps(t) &  M_\eps(t) \end{pmatrix} 
\begin{pmatrix} J\nabla h_\eps\circ\Phi_\eps^t \\ \vec\Lambda_\eps^t\\ \Gamma_\eps^t \end{pmatrix} +
\begin{pmatrix}0\\ \vec M_\eps(t)\\ 0 \end{pmatrix}
\end{equation}
with 
\begin{eqnarray*}
K_\eps(t)&=& {M}_\eps(t)\otimes \Id_{2d} + \Id_{2d} \otimes{M}_\eps(t),\\
C_\eps(t) &=& \begin{pmatrix}\vec C_1(t)^\T\\ \vdots\\ \vec C_{2d}(t)^\T\end{pmatrix}.
\end{eqnarray*}
\end{lem}

\begin{proof}
First we show
$$ {\rm Vec}(AB) = (A\otimes I_{n})\,{\rm Vec}(B)$$
for $A, B\in\R^{n\times n}$. Indeed, for $1\leq cn+i \leq n^2$
\begin{eqnarray*}
{\rm Vec}(AB) _{c n + i} &=& \sum_{j=1}^{n} a_{cj}b_{ji} =  
\sum_{j=1}^{n} (A\otimes I_{n})_{(cn+i),(jn+i)}{\rm Vec}(B)_{jn+i}\\
&=& \sum_{k=1}^{n^2} (A\otimes I_{n})_{(cn+i),k}\,{\rm Vec}(B)_{k} =  \left((A\otimes I_{n})\,{\rm Vec}(B)\right)_{cn+i}. 
\end{eqnarray*}
Similarly one can prove that ${\rm Vec}(BA^{\rm T}) = (I_{n} \otimes A)\,{\rm Vec}(B)$. Hence, 
\begin{eqnarray*}
\partial_t{\rm Vec}(\Lambda_\eps^t) &=& \vec M_\eps(t)+{\rm Vec}\left(
 M_\eps(t)\,\Lambda_\eps^t + 
\Lambda_\eps^t \,M_\eps(t)^{\rm T} \right)\\
&=&\vec M_\eps(t) + ( M_\eps(t) \otimes \Id_{2d})\,\vec\Lambda_{\eps}^t + 
(\Id_{2d} \otimes  M_\eps(t))\, \vec\Lambda_{\eps}^t\\
&=& \vec M_\eps(t) + K_\eps(t)\,\vec\Lambda_{\eps}^t.
\end{eqnarray*}
Moreover, $\tr(C_i(t)\Lambda^t_\eps)=\vec C_i(t)^T \vec\Lambda^t_\eps$ and 
$$
\partial_t\Gamma^t_\eps = M_\eps(t)\Gamma^t_\eps + 
\begin{pmatrix}\vec C_1(t)^\T\\ \vdots\\ \vec C_{2d}(t)^\T\end{pmatrix}\vec\Lambda^t_\eps 
= M_\eps(t)\Gamma^t_\eps + C_\eps(t) \vec\Lambda^t_\eps.
$$
\end{proof}

For Schr\"odinger Hamiltonians $h(q,p)=\tfrac12|p|^2+V(q)$ with smooth subquadratic potential $V:\R^d\to\R$, we observe that
$$
M_\eps(t) = J\,D^2 h\circ\Phi^t_\eps = \begin{pmatrix}0 & \Id_d\\ -D^2 V((\Phi^t_\eps)_q) & 0\end{pmatrix},
$$
while $C_i(t)_{jk}=\partial_k(JD^2 h)_{ij}\circ \Phi^t_\eps$ satisfies $C_i(t)=0$ for $i=1,\ldots,d$ and depends on third derivatives of $V$  and  $(\Phi^t_\eps)_q$ for $i=d+1,\ldots,2d$. This motivates to split 
$$
\Upsilon^t := 
\begin{pmatrix}(\Phi_\eps^t)_q\\ (\Phi_\eps^t)_p\\ \vec\Lambda_\eps^t\\ \Gamma_\eps^t \end{pmatrix} = 
\begin{pmatrix} 0\\ (\Phi_\eps^t)_p\\ \vec\Lambda_\eps^t\\ \Gamma_\eps^t \end{pmatrix} + 
\begin{pmatrix}(\Phi_\eps^t)_q\\ 0\\ 0\\ 0 \end{pmatrix} =:
\Upsilon^t_1 + \Upsilon_2^t
$$
and to rewrite the differential equation (\ref{complete_variational}) as 
\begin{equation}
\label{splitting-equation}
\partial_t\Upsilon^t = A(\Upsilon_2^t)\Upsilon_1^t+ N(\Upsilon_2^t)+B\Upsilon_1^t
\end{equation}
with 
$$
A(\Upsilon_2^t) = \begin{pmatrix}0 &0&0&0\\ 0&0&0&0 \\0&0& K_\eps(t)&0\\ 0&0 & C_\eps(t) &  M_\eps(t)\end{pmatrix},
\quad
B = \begin{pmatrix}0 &\Id_d&0&0\\ 0&0&0&0 \\0&0&0&0\\ 0&0 &0&0\end{pmatrix}
$$
and
$$
N(\Upsilon_2^t) = \begin{pmatrix}0\\-\nabla V((\Phi^t_\eps)_q)\\\vec{M}_\eps(t)\\0\end{pmatrix}.
$$
Let $\phi_a^t$ and $\phi_b^t$ be the flows of the systems 
$$
\partial_t \Upsilon^t = A(\Upsilon_2^t)\Upsilon_1^t + N(\Upsilon_2^t), \qquad
\partial_t \Upsilon^t = B\Upsilon_1^t,
$$ 
respectively, and 
$$
F^h = \phi_a^{h/2} \phi_b^h \phi_a^{h/2} 
$$
the Strang splitting for $h>0$, which provides a second order time discretization of equation \ref{splitting-equation}, 
see e.g. \cite[\S III.3.4]{HWL06}. This splitting scheme produces an approximate solution via 
\begin{eqnarray*}
\Upsilon_1^{h/2} &=& \Upsilon_1^0 + \tfrac{h}{2}(A(\Upsilon_2^0)\Upsilon_1^0 + N(\Upsilon_2^0)),\\
\Upsilon_2^h &=& \Upsilon_2^0 + hB\Upsilon_1^{h/2},\\
\Upsilon_1^{h} &=& \Upsilon_1^{h/2} + \tfrac{h}{2}(A(\Upsilon_2^h)\Upsilon_1^{h/2} + N(\Upsilon_2^h)).
\end{eqnarray*}
and discretizes $\Phi^t_\eps$ by the St\"ormer-Verlet scheme, while $\Lambda_\eps^t$ and $\Gamma_\eps^t$ are discretized by the midpoint rule.

\begin{rem}
The flows $\phi_a^t$ and $\phi_b^t$ are numerically easy to evaluate, since they are determined
by differential equations with constant coefficient matrices and a constant inhomogeneity.
\end{rem}

\section{Numerical experiments}\label{sec:numerics}

Let $\psi_t:\R^d\to\C$ the solution of the Schr\"odinger equation~(\ref{eq:schro}) and $a:\R^{2d}\to\R$ a Schwartz function. The preceeding algorithmical considerations suggest 
\begin{equation}
\label{eq:algo_h2}
\left\langle\psi_t,\opwe(a)\psi_t\right\rangle_{L^2} \approx \frac{1}{N_1} \sum_{j=1}^{N_1} (a_\eps\circ\widetilde\Phi^t_\eps)(w_j) - \frac{\eps}{2\cdot N_2} 
\sum_{j=1}^{N_2} \widetilde \Xi^t_\eps(a)(z_j),
\end{equation}
where $w_1,\ldots,w_{N_1}\in\R^{2d}$ and $z_1,\ldots,z_{N_2}\in\R^{2d}$, $N_1\ge N_2$, are distributed according to $\H^\eps(\psi_0)$. $\widetilde \Xi^t_\eps(a)$ is obtained by $\lfloor t/h_2\rfloor$ iterates of the Strang splitting scheme $F^{h_2}$ and the application of the resulting vector to~$a$ and its derivatives according to Corollary~\ref{cor}. The leading order contribution $\widetilde\Phi^t_\eps$ comes from a sixth order symplectic Yoshida splitting \cite[Table 1, Solution B]{Y90}
with time step $h_1$. 
Our numerical experiments\footnote{All experiments have been performed with {\sc MATLAB}~7.14 on a 3.33 GHz Intel Xeon X5680 processor.} validate this approach for different observables~$a$: 
the position and momentum operators defined by 
$$
(q,p)\mapsto q_j,\quad (q,p)\mapsto p_j,\quad j=1,\ldots,d,
$$
as well as the potential, kinetic, and total energy operators defined by 
$$
(q,p)\mapsto V(q),\quad (q,p)\mapsto\tfrac12|p|^2,\quad (q,p)\mapsto h(q,p),
$$ 
respectively. The main focus of the experiments is on the positive answer of the following two questions: Is the algorithm feasible in a moderate high dimensional setting? Can the asymptotic $O(\eps^2)$ accuracy be acchieved efficiently, that is, with a reasonable number of sampling points $N_1$, $N_2$ and time stepping $h_1$, $h_2$?

\subsection{Six dimensions}\label{sec:henon}

Our first experiment is concerned with the Henon-Heiles potential in six dimensions, 
$$
V(q) = \sum_{j=1}^6\tfrac{1}{2}q_j^2 + \sum_{j=1}^5 \sigma_*(q_jq_{j+1}^2 - \tfrac{1}{3}q_j^3)+ 
\tfrac{1}{16}\sigma_*^2(q_j^2 + q_{j+1}^2)^2
$$
with $\sigma_*=1/\sqrt{80}$. As in \cite[\S5.4]{FGL09} and \cite[\S6]{LR10}, we choose the semiclassical parameter $\eps=0.01$ and a Gaussian initial state (\ref{eq:gauss}) centered in $z_0=(q_0,0)$ with $q_0 = (2,2,2,2,2,2)^{\rm T}$. Since a grid based reference solution of the six-dimensional Schr\"odinger equation is not feasible, we compare the following three asymptotic particle methods: 
\begin{description}
\item[A.] The Husimi based method defined in (\ref{eq:algo_h2}) using $N_1=2^{14}$ and $N_2=2^{10}$ 
transformed Sobol points for the numerical quadratures and the same time stepping $h=10^{-3}$
for both $\widetilde\Phi^t_\eps$ and $\widetilde\Xi^t_\eps(a)$. We note that these numbers of Sobol points provide initial sampling errors smaller than $\eps^2=10^{-4}$ for the leading term, and
considerably smaller than $\eps=10^{-2}$ for the correction term, 
see Table~\ref{tab_henon}.
\item[B.] The naive Husimi method
$$
\left\langle\psi_t,\opwe(a)\psi_t\right\rangle_{L^2} \approx \frac1{N_1} \sum_{j=1}^{N_1}
 (a\circ\widetilde\Phi^t)(z_j),
$$
where $\widetilde\Phi^t$ is a sixth order symplectic Yoshida discretization of Hamilton's equation
$\dot q=p$, $\dot p=-\nabla V(q)$ with time stepping $h$. 
The quadrature nodes $z_1,\ldots,z_N$ are obtained by transforming Sobol points, such that they are distributed according to $\H^\eps(\psi_0)$. This approach is first order accurate with respect to $\eps$.
\item[C.] The second order Wigner based method 
$$
\left\langle\psi_t,\opwe(a)\psi_t\right\rangle_{L^2} \approx \frac1{N_1} \sum_{j=1}^{N_1} 
(a\circ\widetilde\Phi^t)(w_j)
$$
of \cite[\S6]{LR10}, where the quadrature nodes $w_1,\ldots,w_N$ are obtained by transforming Sobol points, such that they are distributed according to the initial Wigner function~$\W^\eps(\psi_0)$. 
\end{description}

Figure~\ref{henon} shows that the difference of the second order algorithms A and~C is bounded by  $6\cdot10^{-4}=6\eps^2$, while the naive Husimi approach of algorithm B deviates from method~C by $0.06=6\eps$. The computing time for the algorithms B and C, which use uncorrected classical transport, is 6~minutes, while the corrected Husimi algorithm~A requires 17 minutes.

\begin{table}[ht!]
\centering
\begin{tabular}{c|cccc}
initial error          & Husimi $N_1$   & Husimi $N_2$ & Wigner $N_1$  & \\ \hline
kinetic energy & $3.2\cdot 10^{-5}$ & $3.3\cdot 10^{-4}$ & $1.6 \cdot 10^{-5}$ &\\
potential energy & $8.1\cdot 10^{-5}$ & $5.8\cdot 10^{-4}$ & $4.8 \cdot 10^{-5}$ &\\
\end{tabular}
\caption{\label{tab_henon} Initial sampling errors of algorithms A (Husimi) and C (Wigner) 
with respect to the analytically 
computed expectation values for the initial kinetic and potential energy of the Henon-Heiles system,
using $N_1=2^{14}$ and $N_2=2^{10}$ transformed Sobol points.}
\end{table}

\begin{figure}[ht!]
\includegraphics[viewport=110 75 300 292,scale=0.8]{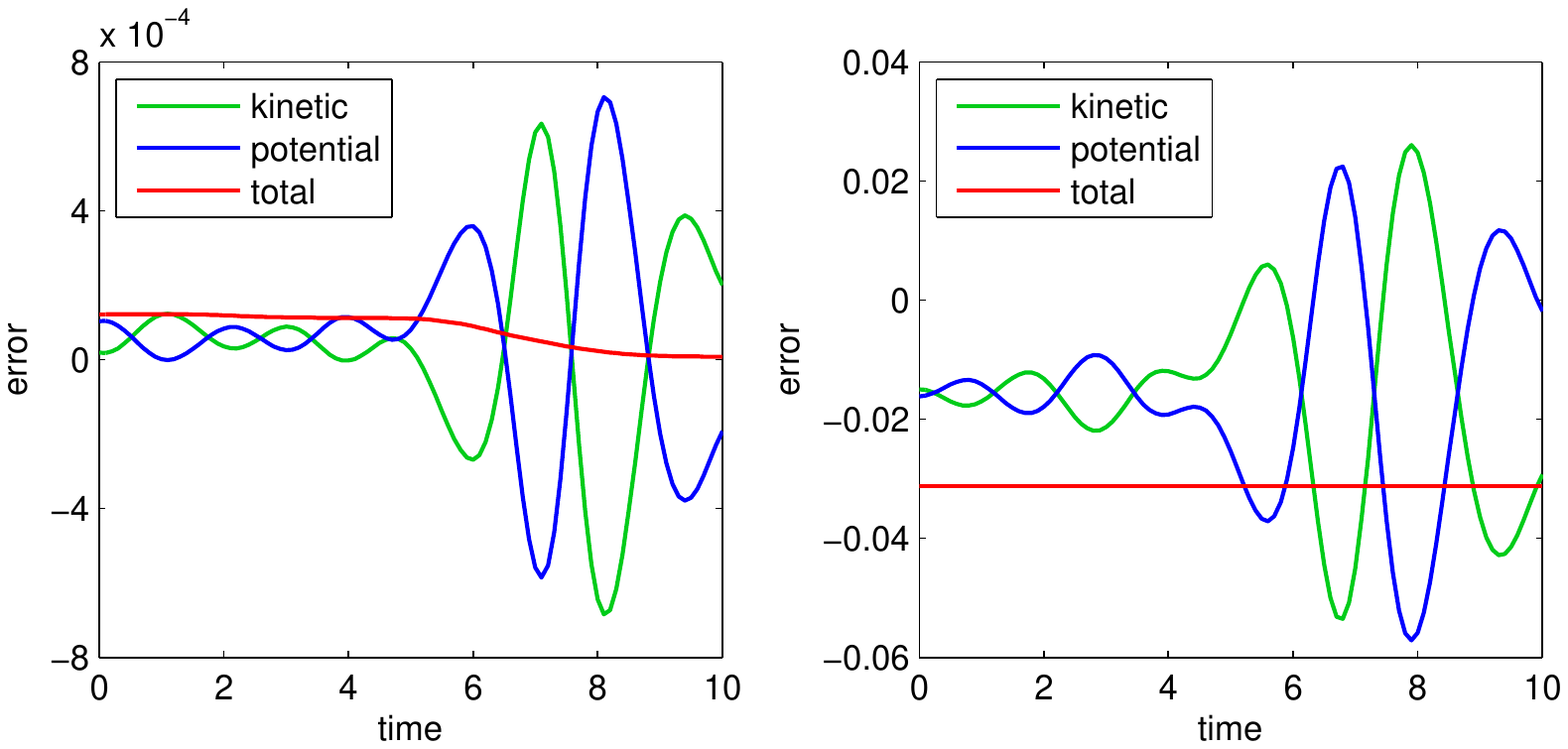}
\caption{\label{henon} The difference of the expectation values of the kinetic, potential and total energy for the six-dimensional Henon-Heiles system ($\eps=0.01$) computed by the second order algorithms A and C (left) as well as B and C (right) as functions of time. The plots illustrate that the corrected Husimi method A is more accurate than the naive Husimi approach B.}
\end{figure}

\subsection{Different semiclassical parameters}

The next set of experiments is performed for the two-dimensional torsional potential, 
\begin{equation}
\label{eq:pot_tor}
V(q) = 2- \cos(q_1)-\cos(q_2),
\end{equation}
and different values of the semiclassical parameter $\eps$. As
initial data we consider both Gaussian states, and superpositions of Gaussian states.

While the Husimi transform of a single Gaussian wave packet is a phase space Gaussian,
for Gaussian superpositions the Husimi function can easily be calculated as 
$$ 
\H^\eps( g_{z_1}+g_{z_2}) = \H^\eps(g_{z_1}) + \H^\eps(g_{z_2}) + 2C_{z_1,z_2},
$$
where the cross term
$$ 
C_{z_1,z_2}(z) =(2\pi\eps)^{-d} \exp\left(-\tfrac{1}{8\eps}|z_-|^2\right) \exp\left(-\tfrac{1}{2\eps}|z- z_+|^2\right) 
\cos\left(\tfrac{1}{2\eps}(c_{1,2}-Jz\cdot z_-)\right)
$$
has a Gaussian envelope centered around the mean $z_+=\frac{1}{2}(z_1+z_2)$ and oscillates with a frequency proportional to the difference $z_-=z_1-z_2$. The shift is $c_{1,2}=q(z_1)p(z_1)-q(z_2)p(z_2)$. The cross term contains a constant damping factor, which is exponentially small in $|z_-|^2$. This allows the tails of $\H^\eps(g_{z_1})$ and $\H^\eps(g_{z_2})$ to absorb the oscillations of the cross term, which turns the Husimi function positive. We perform experiments for the following two different set-ups:

\begin{description}
\item[D.] We apply the 
Husimi based method defined in (\ref{eq:algo_h2}) for Gaussian initial data $g_{z_0}$ centered in $z_0=(1,0,0,0)^\T$.
The considered semiclassical parameters are $\eps\in \{10^{-1},5\cdot 10^{-2},10^{-2},5\cdot 10^{-3},10^{-3}\}$.
\item[E.]  We apply the 
Husimi based method defined in (\ref{eq:algo_h2}). The initial states $\psi_0$ are superpositions 
$g_{z_1}+g_ {z_2}$ with phase space centers $z_1=(0.5,-0.6,0,0)^\T$ and $z_2=(0,1,0,0)^\T$, normalized such
that $\|\psi_0\|_{L^2}=1$.
The considered semiclassical parameters are $\eps\in \{10^{-1},5\cdot 10^{-2},10^{-2},5\cdot 10^{-3},10^{-3}\}$.
\end{description}

For the experiments in set-up D, we use Quasi-Monte Carlo quadrature with transformed Sobol points for all values of $\eps$.
In set-up E, we use two different quadrature methods, depending on the
size of the semiclassical parameter $\eps$.
For $\eps\in \{10^{-1}, 5\cdot 10^{-2}\}$ we generate Markov chains of length $N_1$ and $N_2$ 
by a Metropolis algorithm with jumps between sampling regions centered 
around $z_1$, $z_2$, and $z_+$, see \cite[\S4.1]{KLW09}.
The final results are arithmetic means over ten independent runs, which provides unbiased estimates for the phase space integrals. 
For all $\eps\leq 10^{-2}$, the cross term $C_{z_1,z_2}$ in the Husimi function is 
smaller than $2\cdot10^{-13}$, due to the exponential damping factor $\exp(-\tfrac{1}{8\eps}|z_-|^2)$. Therefore, the cross term is neglected, and the two
Gaussians $\H^\eps(g_{z_1})$ and $\H^\eps(g_{z_2})$ are each sampled by $\tfrac12 N_1$ and $\tfrac12 N_2$
Sobol points for the leading term and the correction, respectively. 

Table~\ref{tab1} summarizes the chosen numbers of sampling points $N_1$, $N_2$ and the time steppings $h_1$, $h_2$ for the discretized flows. We note that the smaller the semiclassical parameter $\eps$, the more accurate the approximation of the corrected Husimi algorithm~(\ref{eq:algo_h2}), the higher the computational cost in terms of sampling 
points and time steps for achieving the asymptotic rate $O(\eps^2)$. However, the constants $\gamma(f)>0$ of the error estimates for Monte Carlo and Quasi-Monte Carlo quadrature decrease with decreasing $\eps$, due to the shrinking variance of the integrands. This observation allows to work with moderate numbers of sampling points, even for small values of the semiclassical parameter.


\begin{table}[ht!]
\centering
\begin{tabular}{c|cccc}
$\eps$          & $N_1$(D)~/~$N_1$(E)          & $N_2$(D)~/~$N_2$(E)        & $h_1$     & $h_2$ \\\hline 
$10^{-1}$       & $10^4$~/ ~$10^5$   & $10^3$~/~$10^4$       & $10^{-2}$ & $10^{-3}$\\
$5\cdot10^{-2}$ & $3\cdot10^4$ ~/~$2\cdot 10^5$   & $3\cdot10^3$ ~/~$2\cdot 10^4$ 
& $10^{-2}$ & $10^{-3}$\\
$10^{-2}$       & $10^5$   & $ 10^4$ & $10^{-3}$ & $10^{-3}$\\
$5\cdot10^{-3}$ & $3\cdot10^5$ & $2\cdot10^4$ & $10^{-3}$ & $10^{-3}$\\
$10^{-3}$       & $10^6$   & $5\cdot10^4$ & $10^{-3}$ & $2\cdot10^{-4}$\\
\end{tabular}
\caption{\label{tab1} The number of sampling points $N_1$ and $N_2$ as well as the time steps $h_1$ and $h_2$ used for the simulations with the second order Husimi algorithm (\ref{eq:algo_h2}) in both set-ups D and E.
For $\eps>10^{-2}$, we need
more sampling points in set-up E than in set-up D due to the slower convergence of the Metropolis quadrature.}
\end{table}

We generate numerically converged reference solutions for the Schr\"odinger equation (\ref{eq:schro}) using a Strang splitting scheme with Fourier collocation for the discretization of the Laplacian. The discretization parameters for the reference solutions, that is, the number of time steps, the computational domain, and the size of the space grid,  are summarized in Table~\ref{tab}. The comparison of the Husimi algorithm~(\ref{eq:algo_h2}) with the expectation values inferred from the reference solutions is presented in Figure~\ref{conv_both}. It confirms our expectations and shows second order accuracy with respect to $\eps$. Moreover, the time evolution of the errors for two of the experiments, set-up D with $\eps=10^{-1}$ and set-up E with $\eps=10^{-3}$, is presented in Figure~\ref{gauss_super}.

\begin{figure}[ht!]
\includegraphics[viewport=100 75 300 320,scale=0.8]{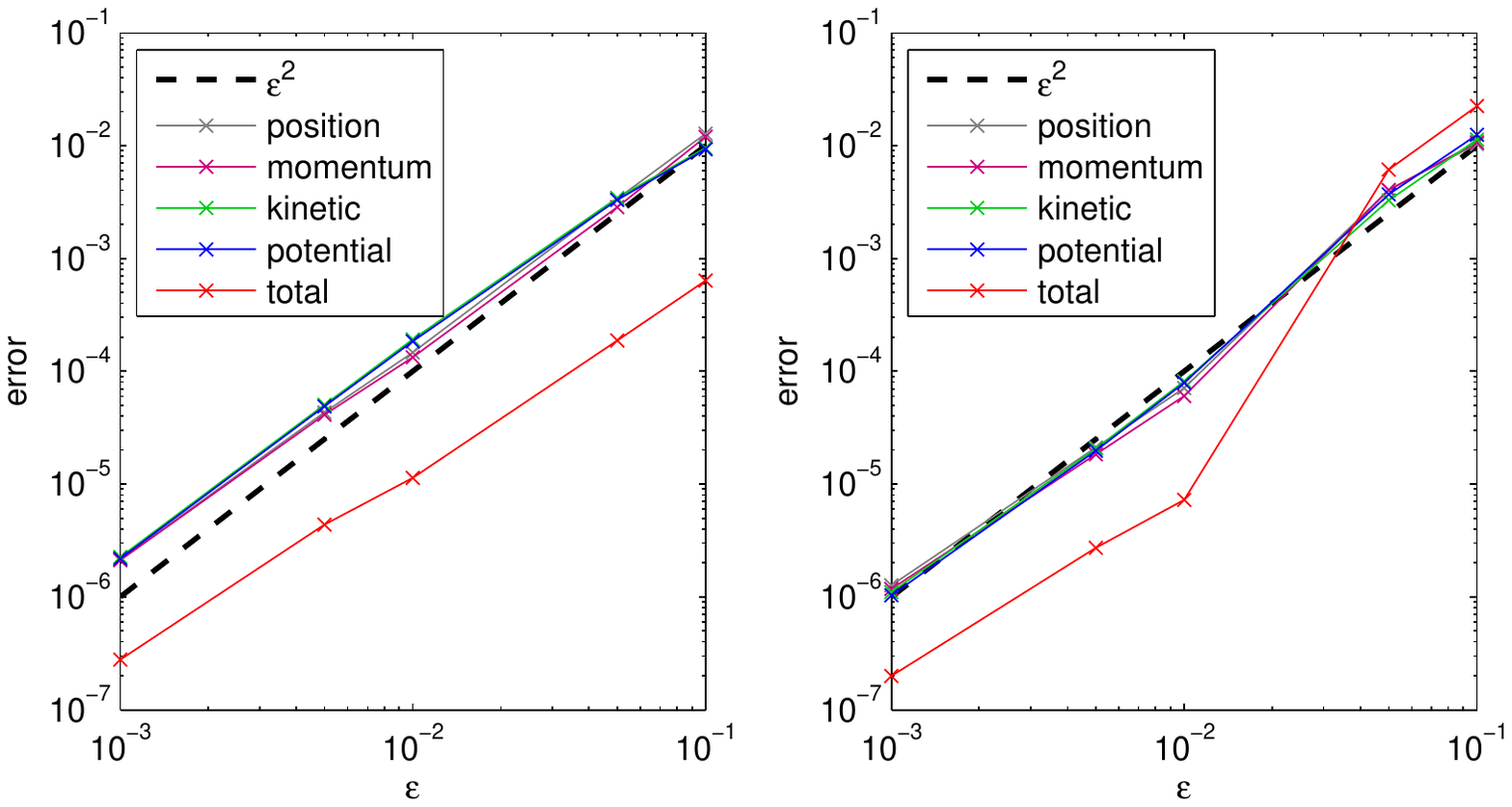}
\caption{\label{conv_both} 
The errors of the expectation values of position, momentum, potential, kinetic, and total energy for the two-dimensional torsional potential computed with the corrected Husimi algorithm~(\ref{eq:algo_h2}) 
in the experimental set-ups D (left) and E (right), averaged over the time interval $[0,20]$. 
In both cases, the errors are of the order $\eps^2$.
}
\end{figure}


\begin{table}[ht!]
\centering
\begin{tabular}{c|ccc}
$\eps$    & $\#$timesteps & domain               & space grid\\\hline 
$10^{-1}$ & $5\cdot10^3$   & $[-3,3]\times[-3,3]$ & $1536\times 1536$\\
$5\cdot 10^{-2}$ & $5\cdot10^3$   & $[-3,3]\times[-3,3]$ & $1536\times 1536$\\
$10^{-2}$ & $7.5\cdot10^3$   & $[-2,2]\times[-2,2]$ & $2048\times 2048$\\
$5 \cdot 10^{-3}$ & $10^4$         & $[-2,2]\times[-2,2]$ & $2048\times 2048$\\
$10^{-3}$ & $10^4$         & $[-2,2]\times[-2,2]$ & $2048\times 2048$
\end{tabular}
\caption{\label{tab} The discretization parameters for the grid-based reference solutions computed by a Strang splitting scheme with Fourier collocation.}
\end{table}

\begin{figure}[ht!]
\includegraphics[viewport=95 92 300 550,scale=0.8]{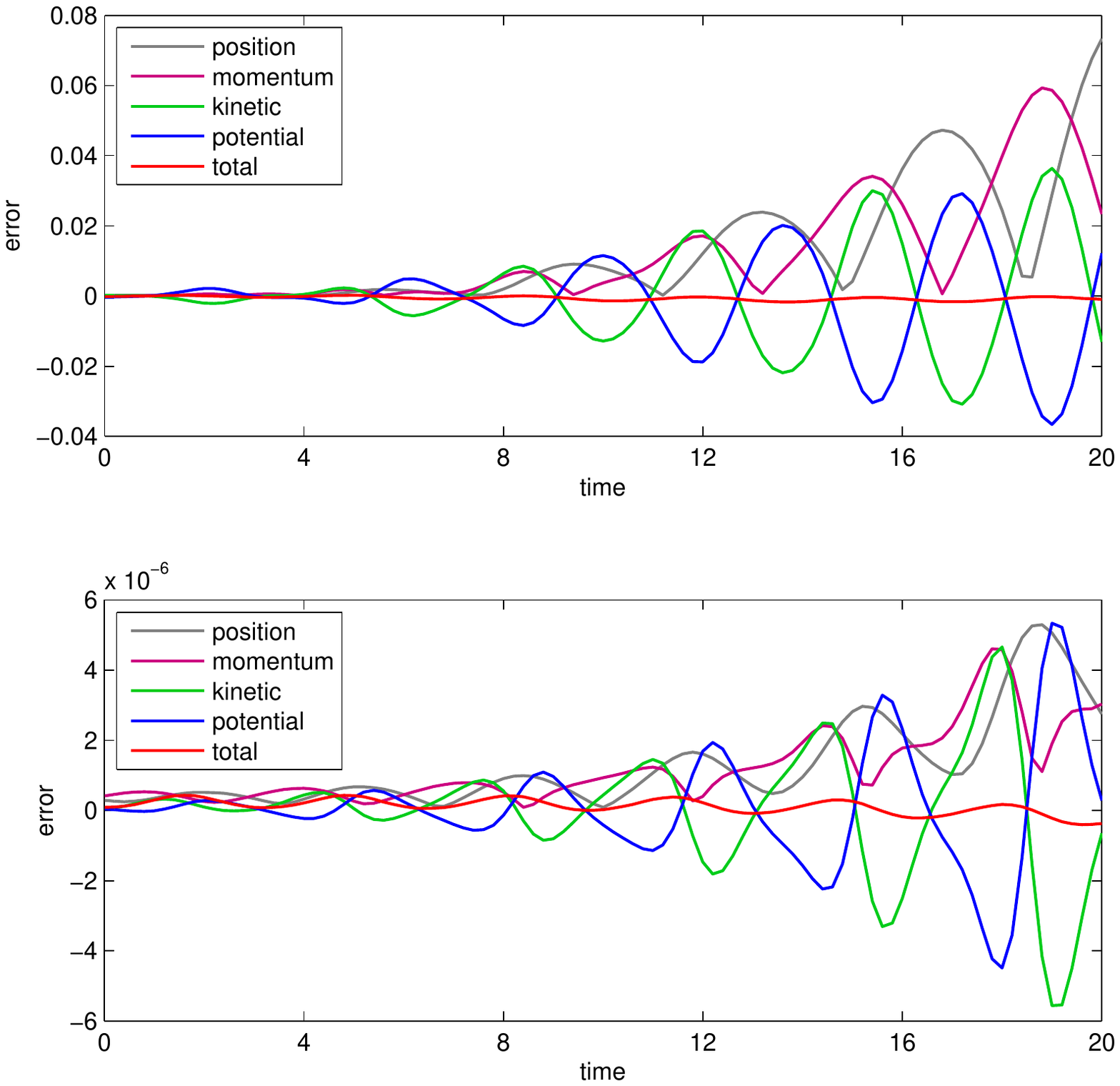}
\caption{\label{gauss_super} 
The errors of the expectation values of position, momentum, potential, kinetic, and total energy
computed with the corrected Husimi algorithm~(\ref{eq:algo_h2}) as functions of time
for experimen~ D with $\eps=10^{-1}$ (top) and experiment~E with $\eps=10^{-3}$ (bottom).
}
\end{figure}

\newpage
\appendix
\section{Composition formula}
\label{anti-wick-composition}

The calculus developed in \S\ref{husimi_functions} implies a composition formula for Anti-Wick quantized operators.

\begin{lem}\label{composition-aw} 
Let $a,b :\R^{2d}\to \R$ be Schwartz functions, $n\in\N$, and $\eps>0$. There is a family of
Schwartz functions $(\gamma_n^\eps(a,b))_{\eps>0}$, $\gamma_n^\eps(a,b):\R^{2d}\to\R$, with $\sup_{\eps>0}\|\opwe(\gamma_n^\eps(a,b))\|_{\mathcal{L}(L^2)}<\infty$, 
such that
\[
\opaw(a)\opaw(b) = \opaw\left(ab+\sum_{m=1}^{n} \eps^m \lambda_m(a,b)\right) + \eps^{n+1}\opwe \left(\gamma_n^\eps(a,b)\right)
\]
with
\[
\lambda_m(a,b)(z) = \frac{(-1)^m}{2^m m!}\Big[\Big( \left(\nabla_z,\nabla_w \right) +  i \left( J\nabla_z, \nabla_w \right)\Big)^m (a(z)\cdot b(w))\Big]_{w=z}~~.
\]
\end{lem}

\begin{proof}
By means of Lemma \ref{AW-Weyl_connection} we can write
\begin{eqnarray*}
\lefteqn{\opaw(a)\opaw(b)}\\
&=&
\opwe \left(\sum\limits_{k=0}^{n}\frac{\eps^k }{4^k k!}\, \Delta^ka \right)
\opwe \left(\sum\limits_{k=0}^{n} \frac{\eps^k }{4^k k!} \, \Delta^kb \right) + 
\eps^{n+1}\opwe\left( v_n^\eps(a,b)\right)
\end{eqnarray*}
for some family of Schwartz functions $(v_n^\eps(a,b))_{\eps>0}$. With the expansion of the Moyal product (\ref{moyal_prod})
 and Lemma \ref{deconvolution} we obtain
\begin{eqnarray*}
\lefteqn{\opaw(a)\opaw(b)}\\
&=&
\opwe \left(\sum_{k+m\leq n} \frac{4^{-(k+m)}\eps^{k+m}}{k!m!} (\Delta^ka) \sharp(\Delta^mb)
\right) + \eps^{n+1}\opwe\left( \widetilde v_n^\eps(a,b)\right)  \\
&=& \sum_{p = 0}^n \eps^{p} \cdot \opwe\left(
\sum_{k+m+j=p} \frac{4^{-(k+m)}(-i)^j}{k!m! j!}\Big[A(D)^j(\Delta^ka\otimes\Delta^mb)\Big]_{\rm diag}  \right) \\
&& + \eps^{n+1}\opwe\left( w_n^\eps(a,b)\right)\\
&=&\sum_{p = 0}^n \eps^{p} \cdot \opaw\left(
\sum_{k+m+j+r=p} \frac{4^{-(k+m+r)}(-i)^j}{ k!m! j! r!}(-\Delta)^r 
\Big[A(D)^j(\Delta^ka\otimes\Delta^mb)\Big]_{\rm diag}  \right) \\
&&+ ~\eps^{n+1}\opwe\left(  w_n^\eps(a,b)\right)
\end{eqnarray*}
with some Schwartz class family $(w_n^\eps(a,b))_{\eps>0}$ by invoking the operator $A(D)$ defined in~(\ref{bidiff_moyal}).
Observing 
\[
\Delta_z \Big[a(z)\cdot b(w)\Big]_{w=z} =  
\Big[ \left(\Delta_z + \Delta_w + 2\left(\nabla_z,\nabla_w \right) \right) \big(a(z)\cdot b(w)\big) \Big]_{w=z} 
\]
yields
\begin{eqnarray*}
\lefteqn{\sum_{k+m+j+r=p}\binom{p}{k,m,j,r}\frac{(-\Delta)^r }{4^r}
\bigg[ (-i A(D))^j\Big(\frac{\Delta_z^k}{4^k}a(z)\cdot
\frac{\Delta_w^m}{4^m}b(w)\Big)\bigg]_{\rm diag}}\\
&=& 2^{-p}\sum_{k+m+j+r=p} \binom{p}{k,m,j,r}\frac{(-\Delta)^r }{2^r}
\Big[ \left(-i\left( J\nabla_z, \nabla_w \right) \right)^j \frac{\Delta_z^k}{2^k} \frac{\Delta_w^m}{2^m}
 (a(z)\cdot b(w))\Big]_{\rm diag}\\
&=&2^{-p}\Big[ \Big(- \left(\nabla_z,\nabla_w \right) - i \left( J\nabla_z, \nabla_w \right)\Big)^p (a(z)\cdot b(w))\Big]_{\rm diag}
\end{eqnarray*}
where we utilized multinomial coefficient notation.
\end{proof}

A proof of Lemma~\ref{composition-aw} in classical scaling can be found in \cite[Theorem 2.5]{AM02}.
The expansion of the composition of two Anti-Wick quantized operators 
from Lemma~\ref{composition-aw} immediately implies a
commutator expansion for Anti-Wick operators, similarly as the Moyal bracket expansion~(\ref{moyal_expansion}) for Weyl quantized operators.

\begin{rem}
For $n=1$, we obtain 
$$
\opaw(a)\opaw(b) = \opaw\left(ab + i\tfrac{\eps}{2} J\nabla a\cdot\nabla b - 
\tfrac{\eps}{2} \nabla a\cdot\nabla b\right) + O(\eps^2),
$$
whose real part coincides with \cite[Lemma 2.4.6]{L10}.
\end{rem}

\bigskip
\paragraph{\bf Acknowledgements} This research was supported by the German Research Foundation
(DFG), Collaborative Research Center SFB-TR 109, and the graduate program TopMath of the Elite Network of Bavaria.

\end{document}